\documentclass{amsart}
\usepackage{oldlfont}
\usepackage{enumerate}
\usepackage{amsmath}
\usepackage{amssymb}
\usepackage{amsthm}
\usepackage{mathrsfs}
\usepackage{amscd}
\usepackage{exscale}
\usepackage{latexsym}
\usepackage[all]{xy}
\usepackage{hyperref}
\usepackage{fancyhdr}
\usepackage{layout}
\usepackage{color}
\usepackage{graphicx}
\usepackage{mathtools}
\usepackage{amsfonts}
\usepackage{calc}
\usepackage{tikz}
\usepackage[OT2,T1]{fontenc}
\usepackage[foot]{amsaddr}

\DeclareMathAlphabet{\mathpzc}{OT1}{pzc}{m}{it}
\DeclareSymbolFont{cyrletters}{OT2}{wncyr}{m}{n}
\DeclareMathSymbol{\Sha}{\mathalpha}{cyrletters}{"58}

\title{Local-global questions for divisibility in commutative algebraic groups}

\author{Roberto Dvornicich}
\address{\emph{Roberto Dvornicich}, Universit\`{a} di Pisa,
Dipartimento di Matematica,
Largo Bruno Pontecorvo 5,
56126 Pisa, Italy} 

\author{Laura Paladino$^*$}
\address{$^*$ \emph{Laura Paladino}, Corresponding author, Universit\`{a} della Calabria,
Dipartimento di Matematica e Informatica,
Ponte Pietro Bucci, Cubo 30B,
87036 Arcavacata di Rende (CS),
Italy, \emph{e-mail address}: laura.paladino@unical.it, ORCID: 0000-0003-4758-9775.}

\begin{document}

\baselineskip=17pt

\numberwithin{equation}{section}

\makeatletter                                                           

\def\section{\@startsection {section}{1}{\z@}{-5.5ex plus -.5ex        
minus -.2ex}{1ex plus .2ex}{\large \bf}}                                 

\def\subsection{\@startsection {subsection}{1}{\z@}{-5.5ex plus -.5ex         
minus -.2ex}{1ex plus .2ex}{\bf}}


\def\@setfoot@addresses{
\def\author##1{}%
 \def\\{\unskip, \ignorespaces}%
 \newif\if@firstaddr
 \@firstaddrtrue
 \def\address##1##2{%
 \if@firstaddr\@firstaddrfalse\else\par\fi
 \@ifnotempty{##1}{(\ignorespaces##1\unskip) }%
 {\ignorespaces##2}%
 }%
 \def\email##1##2{}%
 \def\curraddr##1##2{}%
 \def\urladdr##1##2{}%
 \addresses
 }

\pagestyle{plain}
\markright{  }

\newtheorem{thm}{Theorem}[section]

\newtheorem{mainthm}[thm]{Main Theorem}
\newtheorem*{T}{Theorem 1'}

\newcommand{\ZZ}{{\mathbb Z}}
\newcommand{\GG}{{\mathbb G}}
\newcommand{\Z}{{\mathbb Z}}
\newcommand{\RR}{{\mathbb R}}
\newcommand{\NN}{{\mathbb N}}
\newcommand{\GF}{{\rm GF}}
\newcommand{\divis}{{\rm div}}
\newcommand{\QQ}{{\mathbb Q}}
\newcommand{\CC}{{\mathbb C}}
\newcommand{\FF}{{\mathbb F}}
\newcommand{\PP}{{\mathbb P}}

\newtheorem{lem}[thm]{Lemma}
\newtheorem{cor}[thm]{Corollary}
\newtheorem{pro}[thm]{Proposition}
\newtheorem*{proposi}{Proposition \ref{pro:pro63}}
\newtheorem*{thm_notag}{Theorem}
\newtheorem{Problem}{Problem}
\newtheorem{Question}{Question}
\newtheorem*{problem_non}{Problem}
\newtheorem*{prob1}{Problem 1'}
\newtheorem*{prob2}{Problem 2'}
\newtheorem*{cass}{Cassels' question}

\newtheorem{taggedtheoremx}{Problem}
\newenvironment{taggedproblem}[1]
 {\renewcommand\thetaggedtheoremx{#1}\taggedtheoremx}
 {\endtaggedtheoremx}

\newtheorem{proprieta}[thm]{Property}
\newcommand{\pf}{\noindent \textbf{Proof.} \ }
\newcommand{\eop}{${\Box}$  \relax}
\newtheorem{num}{equation}{}

\theoremstyle{definition}
\newtheorem{rem}[thm]{Remark}
\newtheorem{D}[thm]{Definition}
\newtheorem{Not}{Notation}
\newtheorem{opq}{Open questions}
\newtheorem*{opq_non}{Some open questions}
\newtheorem{Def}{Definition}

\newcommand{\nsplit}{\cdot}
\newcommand{\GGG}{{\mathfrak g}}

\newcommand{\PPP}{{\mathfrak p}}
\newcommand{\GL}{{\rm GL}}
\newcommand{\SL}{{\rm SL}}
\newcommand{\SP}{{\rm Sp}}
\newcommand{\LL}{{\rm L}}
\newcommand{\Ker}{{\rm Ker}}
\newcommand{\la}{\langle}
\newcommand{\ra}{\rangle}
\newcommand{\PSp}{{\rm PSp}}
\newcommand{\GU}{{\rm GU}}
\newcommand{\GO}{{\rm GO}}
\newcommand{\Aut}{{\rm Aut}}
\newcommand{\Alt}{{\rm Alt}}
\newcommand{\Sym}{{\rm Sym}}
\newcommand{\et}{{\rm \acute{e}t}}

\newcommand{\isom}{{\cong}}
\newcommand{\z}{{\zeta}}
\newcommand{\Gal}{{\rm Gal}}
\newcommand{\SO}{{\rm SO}}
\newcommand{\SU}{{\rm SU}}
\newcommand{\PGL}{{\rm PGL}}
\newcommand{\PSL}{{\rm PSL}}
\newcommand{\loc}{{\rm loc}}
\newcommand{\Sp}{{\rm Sp}}
\newcommand{\Id}{{\rm Id}}

\newcommand{\modn}{{\rm \hspace{0.1cm} (mod \hspace{0.1cm} }}

\newcommand{\F}{{\mathbb F}}
\renewcommand{\O}{{\cal O}}
\newcommand{\Q}{{\mathbb Q}}
\newcommand{\R}{{\mathbb R}}
\newcommand{\N}{{\mathbb N}}
\newcommand{\E}{{\mathcal{E}}}
\newcommand{\G}{{\mathcal{G}}}
\newcommand{\Tor}{{\mathcal{T}}}
\newcommand{\A}{{\mathcal{A}}}
\newcommand{\C}{{\mathfrak{C}}}
\newcommand{\bmu}{{\textbf \mu}}
\newcommand{\s}{{\sigma}}

\newcommand\ddfrac[2]{\frac{\displaystyle #1}{\displaystyle #2}}
\newcommand{\longhookrightarrow}{\lhook\joinrel\longrightarrow}

\newcommand\mscriptsize[1]{\mbox{\scriptsize\ensuremath{#1}}}
\newcommand\mtiny[1]{\mbox{\tiny\ensuremath{#1}}}

\vskip 0.5cm

\maketitle

\vskip 1.5cm

\begin{abstract}
\noindent This is a survey focusing on the Hasse principle for divisibility of points in commutative algebraic groups
and its relation with the Hasse principle for divisibility of elements of the Tate-Shavarevich group in the 
Weil-Ch\^{a}telet group. The two local-global subjects arose as a  generalization  of some classical
questions considered respectively by Hasse and Cassels.   We describe the deep connection between the two problems and give an overview of the long-established results and the ones achieved during the last twenty years, when the questions were taken up
again in a more general setting. In particular, by connecting various results about the two problems, we describe how some recent developments in the first of the two local-global questions
imply an answer to Cassels' question, which improves all the results published before about that problem. This answer is best possible over $\QQ$. We also describe some links with other similar questions, as for examples the Support Problem and the local-global principle for existence of isogenies of prime degree in elliptic curves.

\end{abstract}


\section{Introduction} \label{intro}
In 1923-1924 Hasse generalized to all number fields a result shown by Minkowski over $\QQ$.

\par\bigskip\noindent  \textbf{Hasse-Minkowski Theorem.}  
\emph{ Let $k$ be a number field and let $F(X_1,... ,X_n)\in
 k[X_1,... ,X_n]$ be a quadratic form. If $F$ represents $0$ non-trivially in $k_v$, for all completions
$k_v$ of $k$, then $F=0$ has a non-trivial solution in $k$.}

 \par\bigskip\noindent This theorem is also known as Hasse principle on quadratic forms.
The assumption that $F$ is isotropic in $k_v$ \emph{for all but finitely many completions}
(implying the same conclusion) gives a stronger form of the principle. Since then, many mathematicians have been concerned with similar 
so-called local-global problems, i.e. they have been questioning if, given a global field $k$,  the validity of some properties for all but finitely many local 
fields $k_v$ could ensure the validity of the same properties for $k$ (see among others \cite{BaPa}, \cite{CSSD}, \cite{HHK}, \cite{Kato}, \cite{PR}, \cite{Kne}, \cite{Che}, \cite{Poo}).  When the answer to such a problem is affirmative, 
one says that there is a local-global principle or a Hasse Principle. 
Along with the classical Hasse Principle on quadratic forms, one of the most famous local-global principles is the Albert-Brauer-Hasse-Noether's Theorem on central simple algebras, often referred to as Brauer-Hasse-Noether's Theorem (see for instance \cite{Roq}).

\begin{thm}[Albert, Hasse, Brauer, Noether, 1932]
Let $k$ be a number field and let $\mathfrak{A}$ be a central simple algebra over $k$. Then  $\mathfrak{A}$
splits over $k$ if and only if $\mathfrak{A}$ splits over $k_v$, for all places $v$ of $k$.
\end{thm}

\noindent 
Brauer proved that the tensor product equips the set of equivalence
classes of central simple algebras over $k$ with the structure of an abelian group, which is
called the Brauer group of $k$ and is denoted by $Br(k)$ (see the recent monograph \cite{CoSko} written by Colliot-Th\'el\`ene
and Skorobogatov). The cohomological description of $Br(k)$ is $H^2(k,
\bar{k}^*)$, where $\bar{k}$ is the separable closure of $k$, and can be extended in the case of a 
variety $X$ defined over $k$, giving rise to the Brauer-Grothendieck group 
$Br(X)=H^2_{\et}(X,\GG_{m,X})$ (see \cite{CoSko} and  \cite{Sko} for further details). 
In \cite{Man} Manin showed that in many cases the failure of the Hasse principle of the existence
of $k$-rational points on $X$ can
be explained by a reciprocity law imposed by $Br(X)$ on the set 
of adelic points on $X$.
For further details see \cite{Cre3}, in which Creutz shows that the
Brauer–Manin obstruction explains all failures of the Hasse principle of  existence
of $k$-rational points for torsors under
abelian varieties (see also \cite[Th\'eor\`eme 6]{Man}, where a similar statement was proved
under the hypothesis of finiteness of the Tate-Shafarevich group, that we will define in the following).
\par Local-global questions have often an equivalent formulation in terms of principal homogeneous
spaces under some group schemes $\G$ over $k$, that are classified by the first cohomology group
$H^1(k,\G):=H^1(\Gal(\bar{k}/{k}),\G(\bar{k}))$
(see for instance \cite{Hart}, \cite{Sko}). In these cases, when the hypotheses
require that the assertion holds in \emph{all} completions $k_v$, one can study the behaviour of the Tate-Shafarevich group
$\Sha(k,\G)$ to get information about the validity or the failure of the principle. In fact, this group
is the intersection of the kernels of the restriction maps $res_v: H^1(k,\G)\rightarrow H^1(k_v,\G)$, as $v$  varies in the set $M_k$ of places of $k$,
and its vanishing ensures a positive answer to the question.
On the other hand, by answering the problem in some cases, one can get information
about $\Sha(k,\G)$. When the hypotheses of a local-global question require its validity in \emph{all but
finitely many} completions $k_v$, the group that interprets the hypotheses of the problem
in the cohomological context is not exactly $\Sha(k,\G)$, but a similar group, i.e., the intersection of the
kernels of the maps $H^1(k,\G)\rightarrow \prod_{v\in \Sigma} H^1(k_v,\G)$, as $v$ varies in a subset
$\Sigma$ of $M_k$ containing all but finitely many places $v$. The two groups often coincide, but there are some examples
in which they differ (see Section \ref{h1loc_sha} and Section \ref{subsec1} for further details).  In various cases it suffices to study
the behaviour of one of them to understand the structure of the other (see Section \ref{h1loc_sha}).

In this paper we will be concerned with the following local-global problems
and their relation.

\begin{Problem} \label{prob1}
Let $k$ be a number field, $M_k$  the set of the places $v$ of $k$ and $\G$ a commutative and connected algebraic group defined over $k$. Let $P\in {\mathcal{G}}(k)$ and let $q$ be a positive integer. Assume that for all but finitely many  $v\in M_k$, there exists $D_v\in {\mathcal{G}}(k_v)$ such that $P=qD_v$. Is it possible to conclude that there exists $D\in {\mathcal{G}}(k)$ such that $P=qD$?
\end{Problem}

\noindent Problem \ref{prob1} was stated by the first author and Zannier in 2001 \cite{DZ} and it was
named \emph{Local-global divisibility problem}. It is the $r=0$ case of the following problem.

 \begin{Problem} \label{prob2} Let $k$ be a number field, $M_k$  the set of the places $v$ of $k$ and $\G$  a commutative and connected algebraic group defined over 
$k$. Let $q$ be a positive integer, let $\sigma\in H^r(k,\G)$ and let $res_v: H^r(k,\G)\rightarrow H^r(k_v,\G)$ be the restriction map.   Assume that for all but finitely many $v\in M_k$ there exists $\tau_v\in H^r(k_v,\G)$
  such that $q\tau_v=res_v(\sigma)$. 
Can we conclude that there exists $\tau \in H^r(k,\G)$, such that $q\tau=\sigma$?
  \end{Problem}

\noindent In a slightly different form, i.e.\ with the assumption that the local divisibility holds for all $v\in M_k$,
Problem \ref{prob2} was stated in 2016 by Creutz 
\cite{Cre2}. In fact, when $r=1$, a similar question was firstly 
posed by Cassels in 1962 only in the case when $\G$ is an elliptic curve \cite{Cas1}.

\begin{cass}
Let $k$ be a number field and $\E$ an abelian variety of dimension 1 defined over $k$. Are the
elements of $\Sha(k,\E)$ infinitely divisible by a prime $p$ when considered as elements of the Weil-Ch\^{a}telet group
$H^1(k,\E)$ of all classes of principal homogeneous spaces for $\E$ defined over $k$?
\end{cass}

\noindent Here infinitely divisible by $p$ means divisible by $p^l$, for all positive integers $l$.
Thus, if one wonders about the divisibility by every power $p^l$ of $p$ in Problem \ref{prob2}, then
this problem can be considered as a generalization of Cassels' question to all commutative algebraic groups. 
 Both
Problem \ref{prob1} and Problem \ref{prob2} are generally studied in the case when
$q=p^l$, with $p$ a prime number and $l$ a positive integer. In fact, an answer for all powers of prime numbers suffices to
have an answer for a general integer $q$, by using
the unique factorization in $\ZZ$ and B\'{e}zout's identity.

Since 1972, Cassels' question was considered in abelian varieties and not only in elliptic curves,  firstly
by Ba\v{s}makov \cite{Bas1}, \cite{Bas} and in the last few years  by \c{C}iperiani and Stix \cite{CS1}, \cite{CS} and by Creutz \cite{Cre}.

\bigskip In this paper we carefully explain the connection between these problems and between some groups
that interpret the hypotheses of Problem \ref{prob1} and respectively Problem \ref{prob2} in a cohomological context (among them being some Tate-Shafarevich groups). The relation between these groups
was sometimes hinted in the literature, but never explained in details.  We will also give a comprehensive overview of all the results 
achieved for those
questions, with particular emphasis on the case of elliptic curves. In fact, in this last case there is  a recent answer to Problem 1 \cite{PRV2},\cite{PRV3} which implies an affirmative answer to the mentioned Cassels' question for every $p>( 3^{[k:\Q] /2} + 1 )^2$  over a number field $k\neq \QQ$ 
and for every $p\geq 5$ over $\QQ$ (see Theorem \ref{new}). This answer is best possible over $\QQ$. The way of deducing such an answer to Cassels' question (see Section \ref{sec_Tate})
has not been explicitly described in other papers. 
The answer itself for $k\neq \QQ$ has not been explicitly stated in other papers. When $k=\QQ$ it is instead mentioned in \cite{Cre2}
as a consequence of \cite{PRV3}.

\medskip In addition observe that if the point $P$
in the statement of Problem 1 is the zero point in the group law of $\G$ and we require that neither $D$ nor $D_v$, for all but finitely many $v$, is the zero point itself, then
the question can be reformulated as follows: 
\emph{if $\G$ admits a $k_v$-rational torsion point of order $q$, for all but
finitely many places $v\in M_k$, can we conclude that $G$ admits a $k$-rational torsion point of order $q$?}
Thus, the question is somehow related with some other famous problems about torsion points or reductions
of torsion points in abelian varieties, as the \emph{Support Problem} studied by Corrales-Rodrig\'a\~{n}ez and Schoof
in \cite{CRS} or the question studied by Katz in \cite{Katz} about the group of $k$-rational torsion points of an abelian variety. Owing to the connection between the existence of isogenies of prime degree $p$ and the existence of
$k$-rational $p$-torsion points, the question is also linked to the local-global problem for the existence of isogenies of prime degree in elliptic curves, studied by
Sutherland in \cite{Sut}. We will describe these and some other related problems and the main results obtained
about them in Section \ref{sec6}.

\par The paper is structured as follows. At first we give a historical overview of the formulation of the two problems and
their classical solutions. Then we describe a cohomological interpretation for Problem \ref{prob1} and give more details about the
link between Problem \ref{prob1}, Problem \ref{prob2} and Cassels' question, that is discussed in Section 4. Section 5 is dedicated to Problem \ref{prob2} and Cassels' question.  
In the following Section 6 and Section 7 we describe the affirmative results achieved for the three problems and respectively the known counterexamples.
As mentioned above, in the last part of the paper
we illustrate some questions similar or somehow related to the three problems, among them the Support Problem \cite{CRS}, the problem studied by Katz 
about the existence
of a torsion point of a prescribed order \cite{Katz} and the local-global principle for the existence of isogenies
of prime degree \cite{Sut}. We give a brief overview of the main results achieved for those problems too and explain their connections with Problem 1 and Problem 2.

\bigskip\noindent \textbf{Acknowledgments.} A part of this paper was written when the second
author was a guest at the Max Planck Institute for Mathematics in Bonn. She thanks for the hospitality
and the excellent work conditions. The authors are very grateful to Brendan Creutz
for some precious suggestions and for useful discussions. 
They also warmly thank Jacob Stix for helpful discussions. They are grateful to Igor Shparlinski and Boris Kunyavskii for having pointed out 
Remark \ref{density}. Furthermore the authors deeply thank the anonymous referees 
for many valuable comments and suggestions.

\section{Classical problems and classical solutions} \label{sec1}

 \par\bigskip\noindent  In the case of a quadratic form  
 $X^2+rY^2$, where $r$ is a rational number, the Hasse Principle is equivalent to
the statement  ``if a rational number is a square in $k_v$,
 for all but finitely many $v$, then it is a square in $k$''. It  is natural to ask if
such a principle still holds for $q$-powers of rational numbers, where $q$ is a general
positive integer, and not only for rational squares. The answer to such a question was
given by the Grunwald-Wang Theorem (see for example \cite[Chap. IX and Chap. X]{AT}).
Here we state the theorem in its classical form, i.e.\ in the more general case when $k$
is a global field. Through all the paper, for every positive integer $q$, we denote by $\zeta_q$ a primitive
$q$-th root of the unity. Furthermore, for every positive integer $s$, let $\hat{\z}_{2^s}$ be a $2^s$-th root of the unity
such that $\hat{\z}_{2^{s+1}}=\hat{\z}_{2^s}$, and let $\eta_s:=\hat{\z}_{2^s}+(\hat{\z}_{2^s})^{-1}$. In particular,
for every field $k$, there exists an integer $s_k\geq 2$ such that $\eta_{s_k}\in k$,
but $\eta_{s_k+1}\notin k$. 

\begin{thm}[Grunwald-Wang, 1933-1950] \label{GWT}
Let $k$ be a global field, let $q$ be a positive integer and let $\Sigma$ be a set
containing all but finitely many places $v$ of $k$. Consider the group $P(q,\Sigma)=\{x\in k|
x\in k_v^q, \textrm{ for all } v\in \Sigma\}$. Then $P(q,\Sigma)=k^q$ except under the following
conditions:

\begin{enumerate} 
\item[1.] $k$ is a number field;
\item[2.] $-1$, $2+\eta_{s_k}$ and $-(2+\eta_{s_k})$ are non-squares in $k$;
\item[3.] $q=2^tq'$, where $q'$ is odd and $t>s$;
\item[4.] $v\notin \Sigma$, for all $v|2$ where $-1, 2+\eta_{s_k}$ and
$-(2+\eta_{s_k})$ are non-squares in $k_v$.
\end{enumerate}

\noindent In this special case $P(q,\Sigma)=k^q\cup \eta_{s_k+1}^q k^q$.

\end{thm}

\noindent In particular, when $k=\QQ$, the principle for $q$-powers of rational numbers could fail only
when $q$ is divided by $2^t$, with $t\geq 3$. The first example violating the principle was shown by Trost in
1934 (see \cite{Tro}). 

 \begin{thm}[Trost, 1948]
The equation $x^8=16$ has a solution in the $p$-adic field $\QQ_p$, for every $p\neq 2$, but
it has no solutions in $\QQ_2$ and in $\QQ$.
\end{thm}

\noindent Similar examples can be constructed for all powers $2^t$, with $t\geq 3$ and, consequently, for
all integers  $q=2^tq'$, where $q'$ is odd and $t\geq 3$, as in the statement of the theorem.
For further details about the formulation of the Grunwald-Wang Theorem, the reader can see the
survey \cite{Roq} by Roquette.
\par If we denote by $\GG_m$ the multiplicative group over $k$, then the
Grunwald-Wang Theorem holds in the commutative group $\GG_m$ as well
as in $k$.  By questioning if its validity still holds for a general
commutative algebraic group $\G$ instead of $\GG_m$, we get nothing
but Problem 1, i.e.\ the \emph{Local-global divisibility problem in
  commutative algebraic groups}. 
So in the cases when the answer to Problem 1 is
affirmative, we have a kind of a generalization of the Hasse Principle
for squares of $k$-rational numbers.  The answer to the local-global divisibility
depends on $k$ as well as on $q$ and this is already shown by the Grunwald-Wang Theorem
in the case when $\G$ is $\GG_m$.

\par \bigskip As stated in the introduction, the more general Problem 2 also was motivated by a classical problem, i.e.\ Cassels' question.
This question was formulated in 1962 in the third paper of Cassels'  famous series 
\emph{Arithmetic on curves of genus 1} (see  Problem (b) in \cite{Cas4} and Problem 1.2 in \cite{Cas3}; for the
whole series of the mentioned Cassels' papers see  \cite{Cas1}, \cite{Cas2}, \cite{Cas3}, \cite{Cas4}, \cite{Cas5}, \cite{Cas6}).
 An affirmative answer to the local-global divisibility only by $p$ for elements in 
$H^1(k, \E)$ was soon given by Cassels and Tate 
 (see \cite[Lemma 6.1 and its corollary]{Cas4} and see also \cite[Theorem 8.1]{Cas3}). In particular Cassels deduced the validity
of the local-global divisibility by $p$ from the following lemma.

\begin{lem}[Tate, 1962] \label{Tate}
Let $k$ be a number field with algebraic closure $\bar{k}$ and absolute Galois group $G_k:=\Gal(\bar{k}/k)$. 
Let $M$ be a $G_k$-module that is isomorphic to $\ZZ/p\ZZ \times \ZZ/p\ZZ$. Then
an element of $H^2(G_k,M)$ is trivial if it is everywhere locally trivial. 
\end{lem}

\noindent  Here \emph{everywhere locally trivial} means that, for all $v$, the element vanishes in  $H^2(G_{k_v},M(\bar{k_v}))$, where
$\bar{k_v}$ is the algebraic closure of $k_v$ and $G_{k_v}:=\Gal(\bar{k_v}/k_v)$. Assume that $\G$ is a smooth commutative algebraic group and that the 
multiplication-by-$q$ map $[q]$ is \'etale, then we have the exact sequence

$$0\longrightarrow \G[q]\longrightarrow \G  \xrightarrow{\makebox[0.5cm]{{\small $[q]$}}} \G \longrightarrow 0,$$

\noindent where $\G[q]$ is the $q$-torsion subgroup of $\G$, which implies the long-exact sequence of Galois cohomology 

$$ ...\longrightarrow \G(k) \longrightarrow H^1(k,\G[q])\longrightarrow H^1(k,\G)\longrightarrow H^1(k,\G)  \xrightarrow{\makebox[0.5cm]{{\small }}} H^{2}(k,\G[q]) \longrightarrow ... .$$
 
\par\noindent In the case of an elliptic curve $\E$, since $\E[p]\simeq \ZZ/p\ZZ\times \ZZ/p\ZZ$, then the local-global divisibility by $p$ holds in $H^1(k,\E)$, as a consequence of Tate's lemma. On the contrary, for powers $p^l$, with $l\geq 2$, the problem remained open for decades, even in the case  of elliptic curves defined over $\QQ$. In this last case, an affirmative answer for all powers $p\geq 5$ has been lately proved. We will describe it in Section \ref{sec_Tate}, as a consequence of some answers given to Problem 1.

\section{A cohomological interpretation of Problem 1} \label{sec3}

 When  $\G\neq \GG_m$ a useful way to attack Problem 1
was shown in \cite{DZ}, in which the authors gave a cohomological interpretation
of the problem. For every positive integer $q$, we denote by  $K:=k(\G[q])$ the number field generated over $k$ by the coordinates
of the points in the $q$-torsion subgroup $\G[q]$ of $\G$. Since $K$ is the splitting field of the $q$-division polynomials,
then $K/k$ is a Galois extension, whose Galois group we denote by $G$.
Let $P\in \G(k)$ and let $D\in \G(\bar{k})$ be a $q$-divisor of 
$P$, i.e.\ $P=qD$. Let $F$ be the extension of $K$ generated by the coordinates
of $D$. Two $q$-divisors of $P$ differ by a $q$-torsion point of $\G$. Then we have that $F/k$ is a Galois extension
(it is the splitting field of the polynomials whose roots are the coordinates of the points $\tilde{D}\in \G$ satisfying $q\tilde{D}=P$) and we denote by $\Gamma$ its Galois group $\Gal(F/k)$ (see also \cite{DZ}).
For every $\sigma\in \Gamma$, we have

\begin{equation} \label{cocycle_D} q\sigma(D)=\sigma(qD)=\sigma(P)=P. \end{equation}

\noindent Thus the points $\sigma(D)$ and $D$ differ by a point in $\G[q]$ and 
 we can define a cocycle $\{Z_{\sigma}\}_{\sigma\in \Gamma}$ of $\Gamma$ with
values in $\G[q]$ by

\begin{equation} \label{eq1} Z_{\sigma}:=\sigma(D)-D. \end{equation} 

\begin{pro} \label{pro1}
The class of the cocycle $\{Z_{\sigma}\}_{\sigma\in \Gamma}$ defined in \eqref{eq1} vanishes in 
$H^1(\Gamma,\G[q])$ if and only if there exists $D'\in \G(k)$ such that $qD'=P$. 
\end{pro}

\begin{proof}
Assume that $\{Z_{\sigma}\}_{\sigma\in \Gamma}$ vanishes in 
$H^1(\Gamma,\G[q])$, then there exists $W\in \G[q]$ such that
$\sigma(W)-W=Z_{\sigma}=\sigma(D)-D$, for all $\sigma\in \Gamma$. We have 
$\sigma(D-W)=D-W$, for all $\sigma\in \Gamma$. Thus $D':=D-W\in \G(k)$, since $qD'=qD-qW=qD=P$.
The other implication is trivial.
\end{proof}

\noindent 
As mentioned above, the vanishing of some specific first cohomology group often ensures
an affirmative answer to this kind of problems. 
This is quite a standard way of proceeding in
local-global questions, so we stated the proof of Proposition \ref{pro1} 
for the reader's convenience.
The goal in \cite{DZ} was to consider
a subgroup of $H^1(G, \G[q])$, 
whose vanishing still ensures an affirmative answer to Problem 1. 

\begin{D} 
Let $\Sigma$ be a subset of $M_k$ containing all the valuations
$v$ unramified in $K$. For every $v\in \Sigma$, let $G_v:=\Gal(K_w/k_v)$,
where $w$ is a place of $K$ extending $v$.
We call the \emph{first local cohomology group} (of $G$ with values in $\G[q]$) the following subgroup of $H^1(G,\G[q])$:

\begin{equation} \label{h1loc}
H^1_{\textrm{loc}}(G,\G[q]):=\bigcap_{v\in \Sigma} \ker  (H^1(G,\G[q])\xrightarrow{\makebox[1cm]{{\small $res_v$}}} H^1(G_v,\G[q])).
\end{equation}

\end{D}

\noindent The first local cohomology group portrays the hypotheses of the
problem in the cohomological context.  In fact, observe that if there exists a point $D_v\in \G(k_v)$ such
that $P=qD_v$, then as in \eqref{eq1} we can define a cocycle of $G_v$ with values in $\G[q]$
vanishing in $H^1(G_v,\G[q])$.
The elements of $H^1_{\textrm{loc}}(G,\G[q])$ are represented by cocycles that vanish in
$H^1(G_v,\G[q])$, for all $v\in \Sigma$. We can say that the cocycles representing a class in $H^1_{\textrm{loc}}(G,\G[q])$ are \emph{locally coboundaries}. The group $H^1_{\textrm{loc}}(G,\G[q])$ was firstly defined by Tate, as stated by Serre in \cite{Ser},
where the group was introduced (and denoted by $H^1_{*}(G,\G[q])$).
It is very similar to the Tate-Shafarevich group $\Sha(k,\G[q])$ up to isomorphism
(see Section \ref{h1loc_sha} for further details). 
Observe that, by the Chebotarev Density Theorem (see \cite{Tche} and \cite{LO}), the local Galois group $G_v$ varies over all cyclic subgroups of $G$ as $v$ varies in $\Sigma$.
Then, for every $\sigma\in G$, there exists $v\in \Sigma$, such that  $G_v=\langle\sigma\rangle$. 
Thus, if $\{Z_{\sigma}\}_{\sigma\in G}\in H^1_{\textrm{loc}}(G,\G[q])$, then  for every $\sigma \in G$ there exists
$W_{\sigma}\in \G[q]$ such that $Z_{\sigma} =(\sigma-1)W_{\sigma}$.  
As stated in \cite[Definition at pag. 321]{DZ},
we have the following equivalent definition of $H^1_{\loc}(G,\G[q])$.

\begin{D} \label{loc_cond}
 A cocycle $\{Z_{\sigma}\}_{\sigma\in G}\in H^1(G,\G[q])$ satisfies the
\emph{local conditions} if, for every $\sigma\in G$, there exists $W_{\sigma}\in \G[q]$ such that
$Z_{\sigma}=(\sigma-1)W_{\sigma}$. The subgroup of $H^1(G,\G[q])$ formed by all the cocycles satisfying the local conditions
is the first local cohomology group $H^1_{\textrm{loc}}(G,\G[q])$.
\end{D}

\noindent This second definition shows explicitly the kind of cocycles that one has to check 
 if they are coboundaries or not. Such a description was useful to get a solution to the problem
both in cases when the answer is affirmative and in cases when it is negative. In fact, the triviality of the first
cohomology group assures an affirmative answer to Problem 1.

 \begin{thm}[Dvornicich, Zannier, 2001] \label{h1loc} 
If $H^1_{\textrm{loc}}(G,\G[q])=0$, then the local-global divisibility by $q$ holds in $\G$ over $k$.
\end{thm}

\noindent 
On the other hand, in the cases when such a group is nontrivial we have counterexamples over a finite extension of $k$.
In Section \ref{counter1} we state such a converse of Theorem \ref{h1loc} over a finite extension $L$ of $k$ (i.e. Theorem \ref{counter}) and describe its proof, which gives an explicit method to find counterexamples over $L$. In the case of elliptic curves, this method was successfully applied to find counterexamples over $k$ itself. In Section \ref{counter1} we also discuss when one can find a counterexample over $k$
or not.

\begin{rem} \label{density} 
To apply the Chebotarev Density Theorem, it suffices to have a subset of $M_k$ of Dirichlet density 1. So the hypotheses of Problem
\ref{prob1} can be reformulated by asking that the local divisibility holds for a set
of  places $v$ of Dirichlet density 1. 
Indeed we have

$$H^1_{\loc}(G,\G[q])=\bigcap_{v\in S} \ker ( H^1(G,\G[q])\xrightarrow{\makebox[1cm]{{\small $res_v$}}} H^1(G_v,\G[q])),$$

\noindent where $S$ is a subset of $\Sigma$ such that $G_v$ varies
over all cyclic subgroups of $G$ as $v$ varies in $S$.
If we are able to find such a set $S$, then we can replace the hypotheses of Problem \ref{prob1} about the
validity of the local divisibility for all but finitely many $v\in M_k$ with the assumption of the validity of the local 
divisibility for every $v\in S$. Notice that in particular $S$ is finite, being $G$ finite 
(on the contrary $\Sigma$ is not finite). So it suffices to have that
the local divisibility by $q$ holds for a finite number of suitable places to get the global divisibility by $q$. An explicit set
$S$ is produced in \cite{DivPal} for elliptic curves defined over $\QQ$.
\end{rem}

\section{First local-cohomology group and Tate-Shafarevich group} \label{h1loc_sha}

 As stated in the previous sections, the definition \eqref{h1loc} of $H^1_{\textrm{loc}}(G,\G[q])$ 
is very similar to the classical definition of the Tate-Shafarevich group $\Sha(k,\G[q])$ up to isomorphism. 
The Tate-Shafarevich
group was firstly introduced in the case of an abelian variety, but  the definition
can be generalized to the case of a commutative algebraic group $\G$.
We have already defined 

\begin{equation} \label{Sha} \Sha(k,\G):=\bigcap_{v\in M_k} \ker ( H^1(k,\G)\xrightarrow{\makebox[1cm]{{\small $res_v$}}} H^1(k_v,\G)). \end{equation}

\noindent  More generally, for every $r\geq 0$, one can define 

\begin{equation} \label{Sha_r} \Sha^r(k,\G):=\bigcap_{v\in M_k} \ker (  H^r(k,\G)\xrightarrow{\makebox[1cm]{{\small $res_v$}}} H^r(k_v,\G)), \end{equation}

\noindent where $H^r(k,\G):=H^r(G_k,\G(\bar{k}))$ and $H^r(k_v,\G):=H^r(G_{k_v},\G(\bar{k_v}))$. Clearly $\Sha(k,\G)=\Sha^1(k,\G)$. If we consider the $G_k$-module $\G[q]$ instead of $\G$, in the same way we get  

\begin{equation} \label{Sha[q]} \Sha(k,\G[q]):=\bigcap_{v\in M_k} \ker ( H^1(k,\G[q])\xrightarrow{\makebox[1cm]{{\small $res_v$}}} H^1(k_v,\G[q])). \end{equation}

\noindent (and respectively $\Sha^r(k,\G[q])$, with $r\geq 0$). Consider the following modified Tate-Shafarevich group

\begin{equation}  \label{ShaS[q]} \Sha_{\Sigma}(k,\G[q]):=\bigcap_{v\in \Sigma} \ker (H^1(k,\G[q])\xrightarrow{\makebox[1cm]{{\small $res_v$}}} H^1(k_v,\G[q])). \end{equation}

\noindent (see also \cite{San}, in which the author considered similar modified Tate-Shafarevich groups, with a slightly different
notation; in the notation used in \cite{San} the group in \eqref{ShaS[q]} would be denoted by $\Sha_{M_k\setminus \Sigma}$). Clearly $\Sha(k,\G[q])\subseteq \Sha_{\Sigma}(k,\G[q])$ and in particular the triviality
of $\Sha_{\Sigma}(k,\G[q])$ implies the triviality of $\Sha(k,\G[q])$. It is well-known that $H^1_{\loc}(G,\G[q])$ 
is isomorphic to  $\Sha_{\Sigma}(k,\G[q])$ and we have the following Proposition \ref{triviality},
that is proved for instance in \cite[Proof of Lemma 3.3]{Cre12} and \cite[Chapter I, Lemma 9.3]{Mil08}.
We firstly recall that as a consequence of Chevalley's Theorem on the classification
of the commutative algebraic groups in characteristic zero, we have a group isomorphism $\G[q]\simeq (\ZZ/q\ZZ)^n$, 
where $n$ is a positive integer depending only on $\G$ (see \cite{Ser2} and
\cite[\S 2]{DZ}). In the case when $\G$ 
is an abelian variety of dimension $g$, it is well-known that $n=2g$.
Therefore we have a representation of 
$G_k$ in the general linear group $\GL_n(\ZZ/q\ZZ)$

\begin{equation} \label{sim_rho} \rho: G_k \longhookrightarrow \GL_n(\ZZ/q\ZZ). \end{equation}

\noindent The image $\rho(G_k)$ is isomorphic to $G=\Gal(k(\G[q])/k)=\Gal(K/k)$,
and  we still denote by $G$ such an image. The group $G_k$ acts on $\G[q]$ as $G$ and
the $q$-torsion subgroup $\G[q]$ is a $G_k$-module
as well as a $G$-module. We have $G\simeq G_k/\ker(\rho)$ and by the inflation map, the group $H^1(G,G[q])$ 
is isomorphic to a subgroup of $H^1(k,G[q])$.
 Similarly the group
$H^1(G_v,G[q])$ 
is isomorphic to a subgroup of $H^1(k_v,G[q])$, for every $v\in \Sigma$. By the injection given by the inflation map, the group
$H^1_{\loc}(G,\G[q])$ is isomorphic to a subgroup of $\Sha_{\Sigma}(k,\G[q])$.

\begin{pro} \label{triviality}
The groups $  H^1_{\loc}(G,\G[q])$ and $\Sha_{\Sigma}(k,\G[q])$ 
are isomorphic. In particular, if $H^1_{\loc}(G,\G[q])=0$, then $\Sha(k,\G[q])=0.$
\end{pro}

\begin{proof}
Let $\Sigma_{K}$ denotes the set of places $w$ of $K$ extending the places $v\in \Sigma$.
Consider the following diagram given by inflation-restriction exact sequences

\[
\begin{array}{ccccccc}
0 & \xrightarrow{\hspace{0.4cm}} & H^1(G,\G[q])  & \xrightarrow{\hspace{0.1cm}inf\hspace{0.1cm}} &  H^1(G_k,\G[q]) &   \xrightarrow{\hspace{0.1cm}res\hspace{0.1cm}} & H^1(G_K,\G[q]) \\
& & \Big\downarrow \mtiny{\prod res_v}  & &\Big\downarrow \mtiny{\prod res_v} & &\Big\downarrow \mtiny{\prod res_w} \\
0 & \xrightarrow{\hspace{0.4cm}} &\prod_{v\in \Sigma} H^1(G_v,\G[q])  & \xrightarrow{\hspace{0.1cm}inf\hspace{0.1cm}} &  \prod_{v\in \Sigma}H^1(G_{k_v},\G[q]) &  \xrightarrow{\hspace{0.1cm}res\hspace{0.1cm}} &\prod_{w\in\Sigma_{K}} H^1(G_{K_w},\G[q]) \\
\end{array}
\]

\bigskip\noindent The kernel of the vertical map on the left is  $H^1_{\loc}(G,\G[q])$ and the kernel of the
central vertical map is $\Sha_{\Sigma}(k,\G[q])$. The
vertical map on the right is injective because of $G_K$ acting trivially on $\G[q]$ and by
the Chebotarev Density Theorem. Then we have

\[
\begin{array}{ccccccc}
0 &\xrightarrow{\hspace{0.4cm}} & H^1_{\loc}(G,\G[q]) &  \xrightarrow{\hspace{0.1cm}inf\hspace{0.1cm}} & \Sha_{\Sigma}(k,\G[q])   & \xrightarrow{\hspace{0.1cm}res\hspace{0.1cm}} & 0 \\
& & \Big\downarrow  & &\Big\downarrow  & &\Big\downarrow \\
0 & \xrightarrow{\hspace{0.4cm}} & H^1(G,\G[q])  & \xrightarrow{\hspace{0.1cm}inf\hspace{0.1cm}} &  H^1(G_k,\G[q]) &   \xrightarrow{\hspace{0.1cm}res\hspace{0.1cm}} & H^1(G_K,\G[q]) \\
& &\hspace{0.8cm} \Big\downarrow \mtiny{\prod res_v}  & &\hspace{0.8cm}\Big\downarrow \mtiny{\prod res_v} & &\hspace{0.8cm}\Big\downarrow \mtiny{\prod res_w} \\
0 & \xrightarrow{\hspace{0.4cm}} &\prod_{v\in \Sigma} H^1(G_v,\G[q])  & \xrightarrow{\hspace{0.1cm}inf\hspace{0.1cm}} &  \prod_{v\in \Sigma}H^1(G_{k_v},\G[q]) &  \xrightarrow{\hspace{0.1cm}res\hspace{0.1cm}} &\prod_{w\in\Sigma_{K}} H^1(G_{K_w},\G[q]) \\
\end{array}
\]

\bigskip\noindent and  $H^1_{\loc}(G,\G[q])\simeq \Sha_{\Sigma}(k,\G[q])$. In particular, if
$H^1_{\loc}(G,\G[q])=0$, then $\Sha_{\Sigma}(k,\G[q])=0$, implying $\Sha(k,\G[q])=0.$ 
\end{proof}

\medskip\noindent The group $\Sha(k,\G[q])$ itself is often isomorphic to $H^1_{\textrm{loc}}(G, \G[q])$ and 
this is in particular the case when $H^1_{\textrm{loc}}(G,\G[q])$ is trivial, which surely happens for $p$ sufficiently large \cite{Bas}, \cite{Bas1}, \cite{Cre2}. 
Anyway, in a few cases, the two groups may differ. If the local-global principle fails and there is a point locally divisible for all places but a finite number of them (unramified in K) for which the local divisibility does not hold and this point is not globally divisible as well, then there is a nontrivial class in $H^1_{\loc}(G,\G[q])$, whose image in $\Sha_{\Sigma}(k,\G[q])$
does not belong to $\Sha(k, \G[q])$ (see Section \ref{subsec1} for some examples in elliptic curves defined over $\QQ$). 
The most of the results obtained for Problem \ref{prob1} and Problem \ref{prob2}
are produced by showing the triviality of $H^1_{\textrm{loc}}(G, \G[q])$ or $\Sha(k,\G[q])$
\cite{Bas}, \cite{Bas1}, \cite{CS}, \cite{DZ}, \cite{PRV2}, \cite{PRV3}, etc.

\section{On Problem 2 and Cassels' question}  \label{Problem2}

In this section we give some more information about Problem \ref{prob2} and  Cassels' question.
As mentioned above, in the more general case of an abelian variety $\A$, 
 Cassels' question was firstly considered by  Ba\v{s}makov in \cite{Bas1} and \cite{Bas}, since 1972.
Anyway,  even if he stated the problem for abelian varieties, in his papers he focused especially
on elliptic curves. 
In the recent papers \cite{CS1}, \cite{CS}  \c{C}iperiani and Stix gave a very detailed analysis of Cassels' questions,
both in the case of elliptic curves and in the case of general abelian varieties. Quite at the same time than \cite{CS1},
the question for abelian varieties was also considered in \cite{Cre},
in which Creutz  stated the following result (see also \cite[Proposition 4.3]{CS}).

\begin{thm}[Creutz, 2013] \label{thm_cre2} Let $\A$ be an abelian variety defined over a number field $k$. Let $\A^{\vee}$ be its dual
and $\A[q]^{\vee}$  the Cartier dual of $\A[q]$, where $q$ is a positive integer.
In order to have that $\Sha(k,\A)\subseteq q H^1(k,\A)$ it is necessary and sufficient that the image of the natural map
$\Sha(k,\A[q]^{\vee}) \rightarrow \Sha(k,\A^{\vee})$ is contained in the maximal divisible subgroup of $\Sha(k,\A^{\vee}).$
\end{thm}

\noindent The maximal divisible subgroup $\divis (H^1(k,\A^{\vee}))$ of $\Sha(k,\A^{\vee})$ is the right kernel of the pairing

\begin{equation} \label{pair_cre} \Sha(k,\A)\times \Sha(k,\A^{\vee}) \longrightarrow \QQ/\ZZ, \end{equation}

\noindent known as Cassels-Tate pairing, since it was defined by Tate in \cite{Tate171} as a generalization of Cassels’ pairing on the Tate-Shafarevich group of an elliptic curve
(see also \cite{Cre}). 
More generally there is the following Cassels-Tate pairing 

\begin{equation} \label{PT-CT} \Sha^i(k,\A)\times \Sha^{2-i}(k,\A^{\vee}) \rightarrow \QQ/\ZZ, \end{equation}

\noindent whose left and right kernels are the maximal divisible groups of $\Sha^i(k,\A)$ and respectively
 $\Sha^{2-i}(k,\A^{\vee})$, for $0\leq i\leq 2$ \cite[Theorem 8.6.7]{Neu00}, \cite{CS}.  
Observe that if $\Sha(k,\A^{\vee}[q])$ is trivial, then the image of the  map
$\Sha(k,\A[q]^{\vee}) \rightarrow \Sha(k,\A^{\vee})$ is contained in $\divis (H^1(k,\A^{\vee}))$.
On the contrary, the nontriviality of $\Sha(k,\A^{\vee}[q])$ does not assure in general that $\Sha(k,\A)\not\subseteq q H^1(k,\A)$.
The proof of Theorem \ref{thm_cre2}, is based on the Cassels'-Tate pairing.
Consider again the exact sequence

\begin{equation} \label{exact} 0\longrightarrow \A[q]\longrightarrow \A  \xrightarrow{\makebox[0.5cm]{{\small $[q]$}}} \A \longrightarrow 0,\end{equation}

\noindent which implies the long-exact sequence of cohomology 

\begin{equation} \label{exact2} ... \longrightarrow H^r(k,\A[q])\longrightarrow H^r(k,\A)\longrightarrow H^r(k,\A)  \xrightarrow{\makebox[0.5cm]{{\small $\delta_r$}}} H^{r+1}(k,\A[q]) \longrightarrow ... \end{equation}

\noindent  An element  $\sigma\in H^r(k,\A)$ is locally
divisible by $q$ if and only if its image under $\delta_r$ is in $\Sha^{r+1}(k,\A[q])$ and  it is globally divisible by $q$ if and
only if $\delta_r(\sigma)=0$. Therefore, if  $\Sha^{r+1}(k, \A[q])=0$, then the local-global divisibility by $q$ holds in $H^r(k,\A)$.
It is known that  $\Sha^{r+1}(k, \A[q])=0$, for all $r\geq 2$  \cite[Theorem 3.1]{Tate171}. Then Theorem \ref{thm_cre2} implies the following statement too \cite{Cre2}.

\begin{thm}[Creutz, 2016] \label{Creu1}
Assume any of the following:

\begin{enumerate}
\item[1)] $r=0$ and $\Sha(k,\A[q])=0$;
\item[2)] $r=1$ and $\Sha(k,\A[q]^{\vee})=0$;
\item[3)] $r\geq 2$.
\end{enumerate}

\noindent Then the local-global divisibility by $q$ holds in $H^r(k,\A)$.
\end{thm}

\noindent 
Theorem \ref{Creu1} has an extension to the case when $k$ has positive characteristic, that was implemented by Creutz and Voloch in \cite{CV}.  The triviality of  $\Sha(k,\A[p^l])$, for every $l\geq 1$, implies an affirmative answer in $\A^{\vee}$ over $k$ 
to Cassels' question for $p$ and to Problem \ref{prob2} for every power of $p$. When $\A$ is a principally polarized abelian variety, then $\A\simeq \A^{\vee}$. In this last case,
 the triviality of $\Sha(k,\A[p^l])$, for every $l$, is a sufficient condition
to have an  affirmative answer to 
Cassels' question for $p$ in $\A$.

\begin{cor} \label{suff_sha}
Let $\A$ be a principally polarized abelian variety defined over a number field $k$.
If $\Sha(k,\A[p^l])=0$, for some prime number $p$ and some positive integer $l$, then
 the local-global divisibility by $p^l$ holds in $H^r(k,\A)$,  for every $r\geq 0$.
\end{cor}

\noindent In addition, Creutz proved that if $\A$ is  principally polarized  and $\Sha(k,\A)$ is finite, then the condition $\Sha(k,\A[p^l])\neq 0$ implies that for some $r\geq 0$ the local-global divisibility by $p^l$ fails in $H^r(k,\A)$ \cite[Proposition 2.2.]{Cre2}.

\bigskip\noindent 
In \cite{CS} \c{C}iperiani and Stix gave some sufficient conditions to have
$\Sha(k,\A[p^l])=0$, for every $l\geq 0$. 

\begin{thm}[\c{C}iperiani, Stix, 2015] \label{CS}
Let $\A$ be an abelian variety defined over a number field $k$ and let $p$ be a prime number.
If 
\begin{enumerate}
\item[1)]  $H^1(G,\A[p])=0$  and 
\item[2)] the $G_k$-modules $\A[p]$ and $\textrm{End}(\A[p])$ have no common irreducible
subquotient,
\end{enumerate}

\noindent then $$\hspace{0.5cm} \Sha(k,\A[p^l])=0, \textrm{ for every } l\geq 1.$$

\end{thm}

\noindent 
To prove these results and the other ones in their mentioned papers \cite{CS1}, \cite{CS}, the authors use Galois representations,
characters of representations, Poitou-Tate duality
and especially sequences in cohomology and maps between cohomology groups, that allow them to deduce the
triviality of $\Sha(k,\A[p^l])$. The vanishing of this group also assures an affirmative answer to Problem \ref{prob1} by part \emph{1)} of Theorem \ref{Creu1}  (recall that Problem \ref{prob1} is the $r=0$ case
of Problem \ref{prob2} as mentioned in Section \ref{intro}). 
By investigating certain exact sequences
\cite[(4.4) and  (2.1)]{CS} involving the group $\Sha(k,\A[p^l])$ and the group  $\Sha(k,\A)[p^l]$,  i.e. the $p^l$-torsion part of the group $\Sha(k,\A)$, \c{C}iperiani and Stix also get the same conclusion by showing that Theorem \ref{CS} implies the following equality, for all $l$,

$$\{P\in \A(k)| P\in p^lA(k_v), \textrm{ for all } v\in M_k\}=p^l\A(k).$$

\noindent Thus Theorem \ref{CS}  gives sufficient conditions to have an affirmative answer to Problem 1 in abelian varieties.
In view of Theorem \ref{Creu1}, if $\A$ is a principally polarized abelian variety, then Theorem \ref{CS}  gives sufficient conditions to have an affirmative answer to Cassels' question for $p$, to Problem \ref{prob2} for every power of $p$
and to Problem 1 for every power of $p$.  Till now, Cassels' question has not been investigated in a general
commutative algebraic group $\G$.

\section{Known affirmative results about the local-global divisibility problem and Cassels' question} \label{sec_known}

We are going to give an overview of all the results achieved for Problem 1, Cassels' question and Problem 2 since
their formulations. We will also describe in which cases some affirmative results to Problem 1 
implied affirmative results to Cassels' question and to Problem 2, thanks to the connection between the three problems, that
we explained in the previous sections. 

\subsection{Local-global divisibility in elliptic curves}  \label{sec_Tate}
In the case when $\G$ is an elliptic curve $\E$, the local-global divisibility of points 
has been widely studied during the last fifteen years. Having explicit equations
satisfied by torsion points of such commutative algebraic
groups was useful to describe the extension $K/k$
and the group $H^1_{\textrm{loc}}(G,\E[q])$ in
various examples. 
By Tate's Lemma \ref{Tate} and the exact sequence \eqref{exact2} (for $r=1$),  the local-global divisibility by $p$
holds in elliptic curves defined over number fields. This result was reproved in \cite{DZ} and in \cite{Won}.
In \cite{Won} Wong studied
a problem similar to Problem \ref{prob1}, that we will state in  Section \ref{sec6}
(see Problem \ref{probW}).  
\par  Regarding the local-global divisibility by powers $p^l$, with $l\geq 2$, we summarize the results of the main statements of \cite{PRV2} and \cite{PRV3} in the next theorem.

\begin{thm}[Paladino, Ranieri, Viada, 2012-2014] \label{teo_prv}
Let $ p $ be a prime number, $\z_p$ a primitive $p$-th root of the unity and $ \overline{\zeta_p}$ its
complex conjugate.
Let $\E$ be an elliptic curve defined over a number field $ k $  that does not contain the field  $\Q ( \zeta_p + \overline{\zeta_p} )$.
Suppose that at least one of the following conditions holds:
\begin{enumerate}
\item[$(1)$] $ \E $ has no  $k$-rational torsion points of exact order $ p $;
\item[$(2)$] $k ( \E[p] ) \neq k ( \zeta_p )$;
\item[$(3)$] there does not exist any cyclic $k$-isogeny of degree $p^3$ between two elliptic curves defined over $ k $ that  are $k$-isogenous to $\E$.
\end{enumerate}
Then, the local-global principle for divisibility of points by $p^l$ holds in $\E$ over $k$, for all positive integers $l$.
\end{thm}

\noindent  The hypothesis that $k$ does not contain the field  $\QQ(\z_p+\overline{\z_p})$ is necessary, for all the
conditions (1), (2), (3) in Theorem \ref{teo_prv}, as shown by an example produced in
\cite[Section 6]{PRV3}.  The proof is based on showing the triviality of the first local cohomology group by
the use of Galois representations. 
Observe that when $k=\QQ$, in view of Mazur's Theorem on the possible subgroups $\E_{tors}(\QQ)$
of rational torsion points of elliptic curves   \cite{Maz}, condition (1) implies
that the local-global divisibility   by $p^l$, with $l\geq 1$ holds for $\E$ over $\QQ$, 
 for all $p\geq 11$. Furthermore, Merel and Stein \cite{MS} and Rebolledo \cite{Reb} proved that if $\QQ( \E[p] ) \neq \QQ( \zeta_p )$ 
then $p\in \{2,3,5\}$ or $p>1000$. 
Therefore condition (2) implies that the local-global divisibility   by $p^l$, with $l\geq 1$, holds for $\E$ over $\QQ$, 
 for all $p\geq 7$. Finally, in \cite{Ken} Kenku proved that there does not exist any rational isogeny of degree
$5^3$ in elliptic curves over $\QQ$, by showing that the modular curve $Y_0(125)$ has no rational points. 
Then Theorem \ref{teo_prv} implies that the local-global divisibility by $p^l$, with $l\geq 1$, holds in $\E$ over $\QQ$, 
for all $p\geq 5$. 

\begin{cor}[Paladino, Ranieri, Viada, 2012-2014] \label{corPRV}
Let $\E$ be an elliptic curve defined over $\QQ$ and  let $p\geq 5$. 
Then the local-global divisibility by $p^l$, with $l\geq 1$, holds in
$\E$ over $\QQ$. 
\end{cor}

\noindent This result is best possible, since for powers
$p^l$, with $p\in\{2,3\}$ and $l\geq 2$, there are counterexamples,
as we will see in Section \ref{subsec1}. A second proof of Corollary \ref{corPRV} for $p\geq 11$  was given
in \cite{CS} and a third one for $p\geq 5$ was given in \cite[Theorem 24]{LW}. In this last paper Lawson and Wuthrich list all cases when $H^1(G,\E[p^l])\neq 0$
and from them they deduce Corollary \ref{corPRV}. Some of their techniques of proof are similar to the ones in
\cite{PRV2}, \cite{PRV3}, namely the use of the existence of an isogeny of prime degree and Galois representations. 
But they also use some exact sequences in cohomology, that allow them to shorten the proofs. 
\par For a general $k$, condition (1) in Theorem \ref{teo_prv} is also very interesting in view of Merel's Theorem on
torsion points of elliptic curves. Here we recall its
statement \cite{Mer}.

\begin{thm}[Merel, 1994] \label{merel}
For every positive integer $d$, there exists a constant $B(d)\geq 0$
such that for all elliptic curves $\E$ over a number field $k$, with $[k:\QQ]=d$, we have

$$|\E_{tors}(k)|\leq B(d).$$   
\end{thm}

\noindent In his very cited but unpublished paper \cite{Oes}, Oesterl\'e showed that Merel's constant $B ( [k:\Q] )$ can
be taken as $( 3^{[k:\Q] /2} + 1 )^2$. 
Thus Theorem \ref{teo_prv}, combined with Theorem \ref{merel}, implies the next statement.

\begin{cor}[Paladino, Ranieri, Viada, 2012-2014] \label{cor_mer}
Let $ \E $ be an elliptic curve defined over a number field $k$.
Then there exists a number ${B} ( d )$, depending only on the degree $d=[k:\QQ]$
of $ k $ over $\QQ$, such that the local-global principle for divisibility of points by $p^l$ in $\E$ over $k$ holds for 
every prime number $p > {B} ( d )$ and every $l\geq 1$. In addition $B( d ) \leq ( 3^{\frac{d} {2}} + 1 )^2$.
\end{cor}

\noindent  Observe that the statement of Corollary \ref{cor_mer} holds for all $k$ and
not only for number fields that do not contain $\QQ(\z_p+\overline{\z_p})$.
In fact the number ${B} ( d )$ can be chosen as ${\rm max} \{ p_0, ( 3^{d /2} + 1 )^2 \}$,
where $p_0$ is the largest prime such that $ k $ contains the field $\Q ( \zeta_{p_0} + \overline{\zeta_{p_0}} )$.
Since  $p_0 \leq 2 [k: \Q] + 1$, then $B( [k:\Q] ) \leq ( 3^{[k:\Q] /2} + 1 )^2$. 

\par\bigskip Gillibert and Ranieri  considered the restriction of Problem \ref{prob1} to torsion points of elliptic curves over number
fields and showed that in this case the answer is affirmative for all powers of every $p\geq 3$ \cite{GG}. Their result is best possible, since for powers of $2$ the local-global principle fails even in this case.

\par\bigskip In the case of global fields of positive characteristic,  the problem was treated by Creutz and Voloch in the mentioned \cite{CV}. 
In particular they showed some counterexamples to Problem \ref{prob1} in elliptic curves and also some counterexamples to
Cassels' question.
An analogue of Problem \ref{prob1} for Drinfeld modules and respectively Carlitz modules was studied in \cite{Hei} and \cite{NNDQ}.
In \cite{Hei}, van der Heiden considered the problem for Drinfeld modules of rank 1 and rank 2 over function fields. The answer is affirmative in many cases, but there are counterexamples too (see in particular \cite[Theorem 18]{Hei}). In \cite{NNDQ},  Dong Quan Ngoc Nguyen studied a generalization of such a problem for Carlitz modules over function fields and 
generalizes van der Heiden's results by giving some sufficient conditions to have an affirmative answer.

\bigskip Regarding Cassels' question, as showed in Proposition \ref{triviality}, the triviality of
$H^1_{\textrm{loc}}(G,\E[q])$ implies the triviality of
$\Sha(k,\E[q])$. For elliptic curves we have  $\E[q]\simeq \E[q]^{\vee}$ and consequently
$\Sha(k,\E[q])\simeq \Sha(k,\E[q]^{\vee})$. Then, in view of Theorem \ref{thm_cre2}, Theorem \ref{teo_prv} (see also Corollary \ref{corPRV}) assures an affirmative answer to Cassels' question over $\QQ$
for all prime numbers $p\geq 5$. We have also an affirmative
answer to Problem 2, for all $r\geq 0$, for every $q=p^l$, with $p\geq 5$ and $l\geq 1$. 
Furthermore, for a general number field $k$,
Theorem  \ref{teo_prv}, combined with Corollary \ref{cor_mer}, imply an affirmative answer to Cassels' question (respectively to Problem 2, for all $r$) in elliptic curves over $k$,
for every $p> ( 3^{[k:\Q] /2} + 1 )^2$ (resp. for all powers $p^l$ of every  $p> ( 3^{[k:\Q] /2} + 1 )^2$). 
A description of how to reach these conclusions by combining the cited theorems 
does not appear in the literature before. 
The conclusions themselves for $k\neq \QQ$
do not appear in the literature too (the case when $k=\QQ$ is mentioned in \cite{Cre2} as a consequence of \cite{PRV3}).
Here we explicitly resume all these conclusions for Cassels' question in the next theorem.

\begin{thm} \label{new}
Let $ \E $ be an elliptic curve defined over a number field $k$.
Then Cassels' question  has an affirmative answer for all $p > ( 3^{[k:\Q] /2} + 1 )^2$ in $\E$ over $k$. 
 In addition, when $k=\QQ$, Cassels' question has an affirmative answer for all $p\geq 5$.
\end{thm}

\noindent In \cite{CS}, the authors proved a similar result.

\begin{thm}[\c{C}iperiani, Stix, 2015] \label{CS_new}
Let $ \E $ be an elliptic curve defined over a number field $k$. Let $B(d)$ be the constant in Theorem
\ref{merel}, where $d=[k:\QQ]$. Then Cassels' question  has an affirmative answer in the following two cases: 

\begin{enumerate}
\item[1)] $p >\max\{B(d),(2^d+2^{\frac{d}{2}})^2\}$;
\item[2)] $p\geq 3 $, $[k(\z_p):k]\neq 2$ and $\E[p]$ is an irreducible $G_k$-representation.  
\end{enumerate}
\end{thm}

\noindent The bound $p>(3^{[k:\Q] /2} + 1 )^2$ in Theorem \ref{new} is better than the bound $p >(2^{[k:\QQ]}+2^{[k:\QQ]/2})^2$ in Theorem \ref{CS_new}. In fact, we have already observed that $B(d)$ can be taken $\leq (3^{\frac{d}{2}} + 1 )^2$. Then $\max\{B(d),(2^d+2^{\frac{d}{2}})^2\}=(2^d+2^{\frac{d}{2}})^2$. 
Both Theorem \ref{teo_prv} and Theorem \ref{CS_new} 
give  a criterion to establish the validity of Cassels' question for $p$ in $\E$ over $k$.
Observe that if $k\neq k(\z_p)$, then the condition $[k(\z_p):k]\neq 2$ implies $\QQ(\z_p+\overline{\z_p})\not\subseteq k$ in the second part of
Theorem \ref{CS_new}. Moreover the irreducibility of $\E[p]$ as a $G_k$-module implies that $\E$ has no
$k$-rational $p$-torsion points, as required in Theorem \ref{teo_prv}. We do not know if the bound  $p>(3^{[k:\Q] /2} + 1 )^2$ is sharp,
when $k\neq \QQ$. When $k= \QQ,$ such a bound is not sharp, since
Cassels' question has an affirmative answer for all $p\geq 5$; so
we may expect that it can be improved for other number fields too.  The bound $p\geq 5$ is sharp. In fact, when 
 $k=\QQ$,  there
are counterexamples to Cassels' question (resp. Problem 2), for $p\in \{2,3\}$ (respectively for powers $p^l$, with $p\in \{2,3\}$ and $l\geq 2$), see Section \ref{counter_cre} for further details.  When $k=\QQ$, in \cite{CS} the authors give another proof of Theorem \ref{new} for $p\geq 11$.

\subsection{Local-global divisibility in algebraic tori} \label{tori}

The study of Problem 1 in algebraic tori began
in \cite{DZ}.  Illengo gave a more complete description in \cite{Ill}, by proving
the following statement.

\begin{thm}[Illengo, 2008] \label{Ill} Let $\Tor$ be an algebraic torus, defined over $k$, of dimension
$$n< 3(p-1).$$
Then the local-global divisibility by $p$ holds in $\Tor$ over $k$. 
\end{thm}

\noindent Illengo also showed that his bound is best possible, since
for all $n\geq 3(p-1)$ there are counterexamples (see Section \ref{tori2} below).
 For powers of $p$ the question is open. Cassels' question is also open in algebraic tori,
as well as Problem \ref{prob2} for $r\geq 1$.

\subsection{Local-global divisibility in abelian varieties} \label{sec_ab_var}

\bigskip When $\G$ is an abelian variety $\A$, we have some more information about Problem 1, provided by Gillibert and Ranieri in \cite{GR2}.

\begin{thm}[Gillibert, Ranieri, 2017] \label{teo_gab_flo}
Let $\A$ be an abelian variety defined over a number field $k$ and let $p$ be a prime.  Suppose that
there exists an element $\sigma\in \Gal(k(\A[p])/k)$, with order dividing $p-1$ and not fixing
any nontrivial element of $\A[p]$. Moreover suppose that $H^1(\Gal(k(\A[p])/k), \A[p])=0$.
Then the local-global divisibility by $p^l$ holds in $\A$ over $k$ and $\Sha(k,\A[p^l])=0$, for every $l\geq 1$.   
\end{thm}

\noindent The same authors studied the local-global divisibility especially in the case of abelian varieties of $\GL_2$-type \cite{GG},
\cite{GR3}.  

\par\bigskip Along with Theorem \ref{thm_cre2}, one of the main results about Cassels' question in abelian varieties is the mentioned Theorem \ref{CS}, proved by \c{C}iperiani and Stix (see also \cite{CS1}). 
\par Observe that both the conclusion of Theorem \ref{CS} and the conclusion of Theorem
\ref{teo_gab_flo} hold under the assumption that  $H^1(\Gal(k(\A[p])/k), \A[p])$ is trivial. 
As discussed above, the vanishing of this group implies 
the validity of the local-global divisibility by $p$ in $\A$ over $k$ and in $H^1(k,\A)$. 
Therefore the case of the divisibility by $p$ is not covered either by  Theorem \ref{CS}
or by Theorem \ref{teo_gab_flo}. In \cite[Theorem 1.3]{GR2} Gillibert and Ranieri
gave sufficient conditions to have an affirmative answer to Problem 1 in principally polarized
abelian varieties, without assuming $H^1(\Gal(k(\A[p])/k), \A[p])=0$.
Anyway the question for the divisibility by $p$ remained open in abelian varieties
that are not principally polarized, as well as in a general commutative algebraic group.

\subsection{Some remarks about the local-global divisibility of points in other commutative algebraic groups} \label{sec_alg_gr}

\bigskip\noindent We have seen that the local-global divisibility by $p$ holds when $\G$ is a torus isomorphic to $\GG_m$
and when $\G$ is an elliptic curve. This is not true in general for a commutative algebraic group $\G$, 
as shown by some counterexamples that we will describe in next section and as underlined in \cite[Remark 3.6]{DZ}.  In particular they show that for abelian varieties of dimension higher than 1 and for algebraic tori of dimension higher than 1, it is not true in general that the local-global 
divisibility by a prime $p$ holds. Furthermore, we have just observed at the end of the previous section, that  Theorem \ref{CS}  and Theorem
\ref{teo_gab_flo} do not give information about the local-global divisibility by $p$ in
abelian varieties that are not principally polarized. 
A result in \cite{Pal17}, gives conditions on $\G[p]$ ensuring the
validity of the local-global divisibility by $p$, for a general commutative algebraic group $\G$. It underlines that the reducibility of
$\G[p]$ as a $G_k$-module or as an $H$-module, for any subnormal subgroup $H$ of $G_k$
is the greatest obstruction to the local-global divisibility by $p$. In particular
every class of Galois groups $\Gal(k(\G[p])/k)$ for which the local-global divisibility by $p$ may fail in $\G$
is shown in that paper.

\section{Counterexamples} \label{counter1}

\noindent The triviality of $H^1_{\textrm{loc}}(G,\G[q])$ is not exactly
a necessary condition for the local-global divisibility by $q$ in $\G$ over $k$.
In fact, the existence of a cocycle of $G$ with values in $\G[q]$ that satisfies
the local conditions and it is not a coboundary ensures the
existence of a counterexample over a \emph{finite extension} of $k$. Here is
the precise statement, proved in \cite{DZ3}.

\begin{thm}[Dvornicich, Zannier, 2007] \label{counter}
Let $q$ be a positive integer and let $K=k(\G[q])$ be the $q$-division field of a
connected commutative algebraic group $\G$ defined over a number field $k$. 
Let $\{Z_{\sigma}\}_{\sigma\in G}$ be a cocycle with
values in $\G[q]$ representing a nontrivial element in
$H^1_{loc}(G,\G[q])$. Then there exist a number field
$L$ such that $L\cap K=k$ and a point $P\in\G(L)$ which
is divisible by $q$ in ${\G}(L_v)$ for all
places $v$ of $L$, but not divisible by $q$ in $\G(L)$.
\end{thm}

\noindent Therefore, the nontriviality of $H^1_{\textrm{loc}}(G,\G[q])$
is an obstruction to the validity of the principle in finite extensions of $k$. \normalcolor

\noindent A  
method to obtain explicit counterexamples from a nontrivial class in $\{Z_{\sigma}\}_{\sigma\in G}\in H^1_{\textrm{loc}}(G,\G[q])$ 
is given by considering the equality \eqref{eq1} with  $D\in \G(\bar{k})$ (and  $\{Z_{\sigma}\}_{\sigma\in G}$ such a nontrivial element in
$H^1_{\loc}(G,\G[q])$).
 When we know explicit equations for the group law of $\G$, as for instance in the case of elliptic curves, 
 we get an explicit system of equations in the 
coordinates of $D$, as variables. When $\G$ is an elliptic curve, we have
a system of two equations in two variables. In the proof of Theorem \ref{counter} in \cite{DZ3}, the authors show
that, as $\sigma$ varies in $G$, that system defines an algebraic variety
$\mathcal{B}$ that is isomorphic to $\G$ over $K$. Furthermore, they show 
that every $k$-rational point of $\mathcal{B}$ corresponds to a point $D\in \G(K)$, such 
that $P=qD$ is a $k$-rational point of  $\G$ violating the Hasse principle for divisibility by $q$. 
This construction clarifies why in certain cases the non-vanishing of 
$H^1_{\textrm{loc}}(G,\G[q])$ is not a necessary condition; it depends
on the existence of a $k$-rational point on the variety $\mathcal{B}$. In the
case when $\mathcal{B}$ has no $k$-rational points, we are not able to
find a counterexample over $k$. However an $L$-rational point of $\mathcal{B}$, where  
 $L$ is
a finite extension of $k$ linearly disjoint from $K$ over $k$,
corresponds to a point $D\in \G(LK)$ such 
that $P=qD$ is an $L$-rational point of  $\G$ violating the Hasse principle for divisibility by $q$. Theorem \ref{counter}
ensures the existence of an $L$-rational point in $\mathcal{B}$ and consequently assures the existence
of a counterexample to Problem \ref{prob1} in a finite extension of $k$ linearly disjoint from $K$
(in some cases we also have that $L$ is $k$ itself, as stated above).
\par
Once we have a counterexample for $p^l$, a method to find
counterexamples to the local-global divisibility by $p^{l+s}$, for every $s\geq 0$,
is shown by the second author in \cite{Pal3}. It is based on producing maps between $H^1_{\loc}(G, \G[p^l])$
and $H^1_{\loc}(\Gal(k(\G[p^{l+s}]/k), \G[p^{l+s}])$ 
that are injective under certain conditions. 
This method  has been applied  to produce
explicit counterexamples for $2^l$ and $3^l$, with $l\geq 2$, respectively over $\QQ$ and over $\QQ(\z_3)$ (see Section 6.2 for further details). It works both to prove the existence of counterexamples
and to find explicitly some of them.

\begin{rem} \label{rem_qq} The most interesting case for counterexample is when
$k=\QQ$, since a counterexample over $\QQ$ gives
also a counterexample over all but finitely many
number fields $k$. In fact, assume that 
$P$ is a point giving a counterexample to the local-global
divisibility by $q$ in $\G$ over $\QQ$ and let $D$ be a
$q$-divisor of $P$, i.e. $P=qD$. Let   $\QQ(\G[q])(D)$ be the extension of $\QQ$ obtained by adding to
$\QQ(\G[q])$ the coordinates of $D$. 
As stated in Section \ref{sec3}, since two different $q$-divisors of the same point
differ by a $q$-torsion point in $\G$, then $\QQ(\G[q])(D)/\QQ$
is a Galois extension. If $k$ is a number field  linearly
disjoint from  $\QQ(\G[q])(D)/\QQ$ over $\QQ$, then $P$ is locally divisible by $q$
in all but finitely many completions $k_v$, with $v\in M_k$
(because it is locally divisible by $q$ in all but
 finitely many $p$-adic fields $\QQ_p$), but it is not
divisible by $q$ in $k$ (since the coordinates
of the $q$-divisors of $P$ lie in  $\QQ(\G[q])(D)/\QQ$).
\end{rem}

\bigskip Concerning the link between counterexamples to Problem 1 and counterexamples to Problem 2 and Cassels' question,
observe that, by Theorem \ref{thm_cre2}, if we have  counterexamples to Cassels' question in an abelian variety $\A$,
 then the image
of the map $\Sha(k,\A[q]^{\vee}) \rightarrow \Sha(k,\A^{\vee})$ is not contained in the maximal divisible subgroup 
$ \divis (H^1(k,\A^{\vee}))$ of $\Sha(k,\A^{\vee})$. In particular $\Sha(k,\A[q]^{\vee})$ is nontrivial, implying
 $H^1_{\loc}(G,\A[q]^{\vee})\neq 0$ too.
Therefore a counterexample to Cassels' question in $\A$ gives a counterexample to
the local-global divisibility in $\A^{\vee}$, but the converse  is not true.  
In fact, the Tate-Shafarevich group
$\Sha(k,\A^{\vee})$ could vanish even if the first cohomology group 
 $H^1_{\loc}(G,\A[q]^{\vee})$ does not vanish and, in any case, 
even the nontriviality of $\Sha(k,\A[q]^{\vee})$ 
does not imply that the image
of the map $\Sha(k,\A[q]^{\vee}) \rightarrow \Sha(k,\A^{\vee})$ is not contained
in $ \divis (H^1(k,\A^{\vee}))$.  If $\A$ is principally polarized (that in particular happens
when $\A$ is an elliptic curve), we also have that a counterexample to Cassels' question in $\A$ gives a counterexample to
the local-global divisibility in $\A$, but the converse it is not true in general.

\subsection{Counterexamples about Problem \ref{prob1} and Cassels' question in elliptic curves} \label{subsec1}
The first paper dedicated exclusively to the counterexamples to the
local-global divisibility of points of elliptic curves 
is \cite{DZ2}. The authors produced 
explicit counterexamples to the local-global divisibility by $4$
in some elliptic curves over $\QQ$. 
They use  equation \eqref{eq1} and the method explained above. One of the counterexamples is
given by the curve $y^2=(x+15)(x-5)(x-10)$, with its rational point
$P=(1561/12^2,19459/12^3)$, that is locally divisible by $4$ in $\QQ_p$,
for all $p\neq 2$, but it is not divisible by $4$ in $\QQ$ and in $\QQ_2$.
This is one of the cases when $H^1_{\loc}(G,\E[4])$ and $\Sha(k,\E[4])$ are
different, since the point $P$ comes out from a notrivial class $\{Z_{\sigma}\}_{\sigma\in G}$ in $ H^1_{\loc}(G,\E[4])$,
which does not belong to $\Sha(k,\E[4])$, being $P$ not divisible by $4$ in $\QQ_2$. 
In \cite{DZ2} the authors also show that the point 
$P=(5086347841/1848^2,-35496193060511/1848^3)$ of the curve
$y^2=(x+2795)(x-1365)(x-1430)$ is locally divisible by 4
in all $\QQ_p$, but it is not globally divisible by $4$ in $\QQ$. Similar counterexamples appear in \cite{Pal} and in \cite{Cre2}.
As mentioned above, in \cite{Pal3} it was shown that the first cited counterexample 
to the divisibility by $4$, given by $P=(1561/12^2,19459/12^3)$ in $y^2=(x+15)(x-5)(x-10)$,
 can be raised to counterexamples 
to the local-global divisibility by $2^l$, for all $l\geq 2$.
In particular, the point $2^{l-2}P$ gives a counterexample
to the local-global divisibility by $2^l$.\par  
Even when the question about the local-global divisibility is restricted 
only to the torsion points of an elliptic curve, for
$p=2$ there are still counterexamples, as shown by Gillibert and Ranieri in \cite{GG}.

The first counterexamples to the local-global divisibility by $3^l$, for some $l\geq 2$, were produced
in \cite{Pal2}. They are counterexamples to the local-global divisibility by $3^2$,
but the points giving the counterexamples have rational abscissas only, whereas
the ordinates are not rational and are defined over $\QQ(\z_3)$.
Those counterexamples to the local-global
divisibility by $3^2$ imply counterexamples to the local-global
divisibility by  $3^l$, for all $l\geq 2$, in elliptic curves
over $\QQ(\z_3)$ \cite{Pal3}. In 2016 Creutz produced the first
 counterexamples to the local-global
divisibility by $3^l$, for all $l\geq 2$, in elliptic curves
over $\QQ$ \cite{Cre2}. Those examples are
given by the elliptic curve $\E: x^3+y^3+30z^3=0$ defined over $\QQ$
(with distinguished point $P_0=(1:-1:0)$) and the rational point 
$P=(1523698559:-2736572309:826803945)$. For every $l\geq 2$,
the point $3^{l-1}P$ is locally divisible by $3^l$ in all $p$-adic fields $\QQ_p$
but it is not divisible by $3^l$ in $\QQ$. The techniques used to find those counterexamples
are different from the one illustrated above. 
Creutz considered a $3$-covering $C$ of $\E$ over $\QQ$. The 
$3$-coverings of $\E$ over $\QQ$ are parametrized up to isomorphism by $H^1(\QQ,\E[3])$  \cite[Proposition 1.14]{CFO}.
 He took into account the exact sequence 

$$...\rightarrow \E(\QQ)   \xrightarrow{\makebox[0.5cm]{{\small $\delta$}}} H^1(\QQ,\E[3])\rightarrow H^1(\QQ,\E) \rightarrow...$$

\noindent and the class $\xi \in H^1(\QQ,\E[3])$
associated to $C$. The images $res_v(\xi)$ are in $\delta(\E(\QQ_p)[3])$,
for every $p$. On the other hand he took a rational point on $C$ and calculated its
image in $\E$, which is $P$. Therefore $\xi=\delta(P)$ and by showing
$\xi\neq 0$ he got the conclusion. 
 \par Another counterexample to the local-global divisibility by $9$ in elliptic curves over $\QQ$ appears
in \cite{LW}. Lawson and Wuthrich show that the point $(-2,3)$ on the elliptic curve 
$y^2+y=x^3+20$ is locally divisible by $9$ in $\QQ_p$, for all $p\neq 3$,
but it is not divisible by $9$ in $\QQ$ and in $\QQ_3$. To find this counterexample, they
considered elliptic curves admitting a 3-isogeny where either the kernel has a rational 3-torsion point or 
the kernel of the dual isogeny has a rational 3-torsion point. In fact in \cite{DZ3} it is essentially shown that
the non-existence of an isogeny of degree $p$ for $\E$ is a sufficient condition to the local-global
divisibility of points by $p^l$, for every $l\geq 1$ over fields $k$ not containing
$\QQ(\z_p+\overline{\z_p})$ (indeed the proof is carried on in the case when $k=\QQ$, but it can be easily
generalized to every number field $k\not\supseteq\QQ(\z_p+\overline{\z_p})$; see also \cite{PRV2} where this last condition has been explicitated). For each of those $\E$, Lawson and Wuthrich computed
the pair $(a_p(\E), p)$ modulo 9, for all $p< 1000$, where $a_p(\E)$ is the trace of the Frobenius element of $\Gal(k(\E[p])/k)$. 
On the other hand, they considered the pairs $(\textrm{tr}(\sigma), \det(\sigma))$, formed by
trace and determinant of matrices $\sigma\in \GL_2(\ZZ/9\ZZ)$, such that the kernel 
of the map $H^1(G,\E[9])\rightarrow \prod_{v\in \Sigma} H^1(G_v, \E[9])$
is nontrivial. In the cases when
a pair of the first type coincided with a pair of the second type, they got a
curve $\E$, which was a good candidate for  a counterexample.
Then they checked that for
the candidate with the smallest conductor, the local divisibility by 9 
holds for all but finitely many $v$, but the global divisibility does not hold.

 \par\bigskip There are no explicit counterexamples to the local-global
divisibility by powers of $p\geq 5$ in elliptic curves over any number field $k$. Anyway, 
in \cite{Ran} Ranieri exhibited all possible subgroups $G$ of $\GL_2(\ZZ/p\ZZ)$,
such that there could exist an elliptic curve $\E$ with Galois group $\Gal(k(\E[p])/k)$ isomorphic to
$G$, with a point violating the local-global divisibility by $5^l$, for some
positive integer $l$.

\subsubsection{Counterexamples to Problem \ref{prob2} and Cassels' question in elliptic curves} \label{counter_cre}

The first counterexample to 
Problem \ref{prob2} appears in \cite{Cre} and it is given by the curve 

$$y^2=x(x+80)(x+205),$$

\noindent such that $\Sha(\QQ,\E)\not\subseteq 4H^1(\QQ,\E)$. 

 To produce this example Creutz used Theorem \ref{thm_cre2}.
In \cite{Cre2}, he showed that this example implies counterexamples to 
Problem \ref{prob2} in elliptic curves, for  every power $2^n$, with $n\geq 2$.
In the same paper he also shows  counterexamples to 
Problem \ref{prob2} in elliptic curves for  every power $3^n$, with $n\geq 2$.
Those counterexamples also imply a negative answer to Cassels' question for $p\in \{2,3\}$.

\subsection{Counterexamples to Problem 1 in algebraic tori} \label{tori2}

As mentioned in Section \ref{tori}, Illengo showed that the bound in Theorem
\ref{Ill} is best possible, since for all $n\geq 3(p-1)$ there are counterexamples.

\begin{thm}[Illengo, 2008]
Let $p\neq 2$ be a prime and let $l\geq   3(p-1)$. Let $\FF_p^l$ be the field with $p^l$ elements.
There exists a $p$-group
$G$ in $\SL_n(\ZZ)$ such that the map $H^1(G,\FF_p^l)\rightarrow \prod H^1(C,\FF_p^l)$,
where the product is taken on all cyclic subgroups $C$ of $G$, is not injective.
\end{thm}

\noindent Other counterexamples to the local-global divisibility by
$p$ in algebraic tori are produced in \cite{DZ}.

\subsection{Counterexamples to Problem \ref{prob1}, Problem \ref{prob2} and Cassels' question in abelian varieties} \label{cass_cs}

\par In 1996 in \cite[Chapter 6, \S9, pag. 61]{CF} Cassels and Flynn produced a counterexample to the
local-global principle for divisibility by $p = 2$ in an abelian surface. They 
considered a curve $\mathcal{C}$ of genus 2 defined over $\QQ$, with equation
$y^2=A(x)B(x)C(x)$, where $A(x),B(x),C(x)\in \QQ[x]$ are quadratic polynomials with constant term equal to 1,
irreducible over $\QQ$, but splitting respectively over $\QQ(\sqrt{2})$, $\QQ(\sqrt{17})$ and $\QQ(\sqrt{34})$.
They showed that the Jacobian $\A$ of $\mathcal{C}$ has a point $P$ locally divisible by $2$ over all p-adic fields $\Q_p$ and over $\RR$, but not divisible by 2 over $\QQ$.
This is a counterexamples to Problem 1 (and to the $r=0$ case of Problem 2) in
abelian varieties and predates the counterexamples in elliptic curves. Since the local divisibility holds for all $p$, then
we have $\Sha(\QQ,\A)\neq 0$. The abelian surface $\A$ is principally polarized and it is conjectured that  $\Sha(\QQ,\A)$ is finite.
In this last case we also have a counterexample to the
local-global disivibility by 2 in $H^r(\QQ,\A)$, for some $r$, as mentioned in Section \ref{Problem2}.
 More generally, in \cite{Cre} Creutz showed some counterexamples to Problem 2 and to Cassels'
question in abelian varieties for every $p$. 

\begin{thm}[Creutz, 2013] \label{thm_13} Let  $k=\QQ(\z_p)$. Let $p, r$ be two prime numbers
satisfying $r\equiv 1 \modn p^2)$ if $p$ is odd or $r\equiv 1 \modn 8)$ if $p=2$.  Let

\begin{equation} f(x)= (x^p-\z_p)(x^p-r)(x^p-\z_pr)\dots (x^p-\z_p^{p-1}r). \end{equation} 

\noindent There are infinitely many
classes $c\in k^*/(k^*)^p$ such that the Jacobian $J$ of the cyclic cover of the projective line $\PP^1(k)$
defined by $y^p=cf(x)$ satisfies $\Sha(k,J)\not\subseteq p H^1(k,J)$. In particular there are infinitely many
non-isomorphic abelian varieties over $k$ with this property.
\end{thm}

\noindent 
Those counterexamples are over the cyclotomic field $k=\QQ(\z_p)$, but 
the author also deduces counterexamples over $\QQ$, by restriction of scalars.
\noindent To have counterexamples to Problem 1 for the divisibility by $p$ in the Jacobian $J$ of Theorem  \ref{thm_13},
one can take the class $c = 1$.

\section{Other related problems} \label{sec6}

In the literature there are various classical problems and also recent ones  somehow linked
to  Problem \ref{prob1} and Problem \ref{prob2}. We are going to recall briefly some of them.
As already observed in the introduction, if the point $P$ 
in the statement of Problem 1 is the zero point in the group law of $\G$ and we ask that neither $D$ nor $D_v$, for all but finitely many $v$, 
is the zero point itself, then
the question can be reformulated as follows.
 
\begin{Question} \label{probP} If $\G$ admits a $k_v$-rational torsion point of order $q$, for all but
finitely many places $v\in M_k$, can we conclude that $G$ admits a $k$-rational torsion point of order $q$?
\end{Question}

\noindent
This is one of the reasons why the local-global divisibility problem is related to the following problems about torsion points, 
number fields generated by the coordinates of torsion points, existence of isogenies in abelian varieties, etc.
Through all this section, we denote by ${{\mathfrak{p}}}_v$ the prime ideal associated to the valuation $v$
and by $\FF_{v}$ the residue field.

\begin{enumerate} 

\item[i.] In \cite{Katz}  Katz studied this problem formulated by Lang.

\begin{Problem}[Lang, 1981] \label{probK}
 Let $q\geq 2$ be a positive integer.
Let $\A$ be an abelian variety defined over a number field $k$ and let $A_{tors}(k)$ denote
the set of $k$-rational torsion points of $\A$.  For every $v\in M_k$, let $\tilde{\A}_v$ be the reduction of $\A$ modulo
$v$ and let $N(v)$ denote the
number of $\FF_v$-rational points of $\tilde{\A}_v$.  Suppose that the congruence 
$$N(v)\equiv 0 \modn q)$$ holds for a set of places $v$ of Dirichlet density 1. Does there exist an abelian variety
$A'$ that is $k$-isogenous to $A$ and such that
$$\# A'_{tors}(k) \equiv 0 \modn q)?$$ 
\end{Problem}

\noindent  Katz proved that when $\A$ is an elliptic curve or an abelian variety of dimension 2, then we have an affirmative answer to Problem \ref{probK}. 
On the contrary, he produces counterexamples for every abelian variety of dimension $g\geq 3$
and every positive integer $q\geq 3$. 
\par Assume that $\E$ is an elliptic curve with good reduction at $v$.  If $q$ is coprime with the characteristic of the residue field $\FF_{v}$, then by \cite[VII. Proposition 3.1]{Sil}, the group $\E(k_v)[q]$ of the $k_v$-rational $q$-torsion points of $\E$ injects into the group $\tilde{\E}_v(\FF_v)$ 
of the $\FF_v$-rational points of $\tilde{\E}_v$.  Then the existence of a $k_v$-rational point of exact order $q$ implies the existence of
a subgroup of $\tilde{\E}_v(\FF_v)$ of order $q$. If $q\nmid N(v)$, then there are no $k_v$-rational $q$-torsion points in
$\E$. Since $\E(k)[q]$ injects into $\E(k_v)[q]$, this also implies that there are no $k$-rational $q$-torsion points in $\E$. 
Therefore an affirmative answer to Problem \ref{probK} implies an affirmative answer to Question \ref{probP}. 
\par Katz reformulated Lang's question in term of representations to prove some of his results  (see also \cite{Cul}, \cite{Cul2}).

\begin{problem_non}[Lang, Katz, 1981]  \label{probK2}
 Let $q\geq 2$ be a positive integer.
Let $\A$ be an abelian variety defined over a number field $k$. For every place $v$,
denote by $T_v(\A)$ the $v$-adic Tate module of $\A$, by $\rho_v$ the associated $v$-adic
representation and by $\bar{\rho_v}$ the associated mod $v$ representation. 
If for every $\sigma\in G_k$, we have $\det(1-\bar{\rho}_v(\sigma))= 0$ in $k_v$, is it true that
the semisimplification of   $T_v(\A)\otimes k_{v}$ contains the trivial representation?
\end{problem_non}
\bigskip

 \item[ii.] Owing again to the particular case when $P\in \G(k)$ is the zero point, 
 Problem 1
 is also linked to the famous \emph{Support Problem}, considered
 by  Corrales-Rodrig\'a\~{n}ez and Schoof
in \cite{CRS}.  The original question about integers was 
posed by Erd\H{o}s in 1988,   during a conference in number theory that took place in Banff.

\begin{Problem}[Support Problem,  Erd\H{o}s 1988] \label{ProbE}
Let $x,y,q$ be positive integers such that $x^q\equiv 1 \modn p)$ if and only if 
 $y^q\equiv 1 \modn p)$, for every prime number $p$. Can we conclude that
$x=y$?
\end{Problem}

\noindent The name \emph{Support Problem} is a consequence of the name \emph{support} used to indicate
the set of prime numbers dividing $x^q-1$. In \cite{CRS}, Corrales-Rodrig\'a\~{n}ez and Schoof
answered affirmatively to Problem \ref{ProbE}.  Moreover they show that the answer is affirmative even with the
hyphoteses holding for all but finitely many prime numbers $p$. They also considered 
the question on a number field $k$ and prove the
following statement.

\begin{thm}[Corrales-Rodrig\'a\~{n}ez, Schoof, 1997] 
Let $k$ be a number field and let $x,y\in k^*$. Assume that for almost all
valuations $v$ of $k$, and for all positive integers $q$ one
has 
$$y^q\equiv 1 \modn {{\mathfrak{p}}}_v) \hspace{0.3cm} whenever \hspace{0.3cm} x^q\equiv 1 \modn {{\mathfrak{p}}}_v).$$
Then $y$ is a power of $x$.
\end{thm}

\noindent In addition, they considered the same question in the case of elliptic curves.

\begin{Problem}  [Corrales-Rodrig\'a\~{n}ez, Schoof, 1997] \label{ProbS}
Let $\E$  be an elliptic curve over a number field $k$. Let $P, Q\in \E(k)$. Assume that for every positive integer $q$
and all but finitely many places $v$ of $k$ for which $\E$ has good reduction, we have

$$qP\equiv 0  \textrm{ in } \FF_v  \hspace{0.3cm}  whenever \hspace{0.3cm} qQ\equiv 0  \textrm{ in } \FF_v.$$

\noindent What can we conclude about $P$ and $Q$?  
\end{Problem}

\noindent One of the main differences between Question 1 or Problem \ref{probK} and the Support Problem is that in this last case one considers  \emph{two} points having the same behaviour with respect a certain property
 and wonders about their possible relation.
 
\par
\noindent Corrales-Rodrig\'a\~{n}ez and Schoof showed that two points $P$ and $Q$ satisfying the assumptions
of Problem \ref{ProbS} are actually linked as follows.

\begin{thm}[Corrales-Rodrig\'a\~{n}ez, Schoof,1997] \label{teoS}
Let $\E$ be an elliptic curve defined over a number field $k$. Let $P, Q\in \E(k)$. If for every positive integer $q$
and all but finitely many places $v$ of $k$ for which $\E$ has good reduction, one has

$$qP\equiv 0  \textrm{ in } \E(\FF_v)  \textrm{  whenever } qQ\equiv 0  \textrm{ in } \E(\FF_v),$$

\noindent then either $Q=\phi(P)$, for some $k$-rational endomorphism $\phi$ of $\E$ or both $P$ and $Q$ are
torsion points.
\end{thm} 

The same question was afterwards considered for abelian varieties by Larsen \cite{Lar}, by Demeyer and Perucca \cite{DP} \cite{Peru} \cite{Per3}, by Banaszak, Gajda and Kraso\'n \cite{BGK} \cite{BGK2} and by Baran\'czuk \cite{Bar}. In particular in \cite{Lar} Larsen proved this generalization of Theorem \ref{teoS}.

\begin{thm}[Larsen, 2003] \label{teoL}
Let $\A$  be an abelian variety over a number field $k$. Let $P, Q\in \A(k)$. Assume that for every positive integer $q$
and all but finitely many places $v$ of $k$ for which $\A$ has good reduction, we have

$$qP\equiv 0  \textrm{ in } \E(\FF_v) \hspace{0.2cm} \Longrightarrow \hspace{0.2cm} qQ\equiv 0  \textrm{ in } \E(\FF_v).$$

\noindent Then there exists a $k$-endomorphism $\phi$  of $\A$ and a positive integer $m$  such that $\phi(P)=mQ$.  
\end{thm}

\noindent In \cite{DP}  Demeyer and Perucca showed an explicit $m$. In the same paper, as well as in \cite{Per3} they also considered the question for tori. Moreover Li treated it for Drinfeld modules  in 
\cite{Li}. In \cite{KP} Khare and Prasad studied the same local-global problem for endomorphisms of an abelian variety (and, more generally, of a commutative algebraic group) instead of points.  Other similar problems are treated in \cite{Kow} and in \cite{AR}.

\bigskip
 \item[iii.]  In \cite{Won}, Wong considered the following question.

\begin{Problem}[Wong, 2000] \label{probW}
Let $\G$ be an algebraic group defined over a number field $k$ and $q>1$ a positive integer. 
Denote by $\Lambda$ a subset of $M_k$ of density 1 and by
 $U$  a finite subset of the set of $k$-rational points $\G(k)$ of $G$. For every $v\in M_k$ let
$\tilde{\G}_v$ be the reduction of $\G$ modulo $v$.  
Assume that for every $v\in\Lambda$, there exists a non-zero point $P_{v}\in U$,
whose image in $\tilde{\G}_v(\FF_{v})$ is a $q$-th power of a point in $\tilde{\G}_v(\FF_{v})$.
Does $U$ contains a $q$-th power of an element of $\G(k)$?
\end{Problem}

\noindent The answer clearly depends on $U$, as well as on $k$ and $q$.
When $U=\{P\}$, with $P$ a $k$-rational point of $\G$,  Problem  \ref{probW} is similar to Problem \ref{prob1}, 
but here one considers the $q$-divisors
of the image of $P$ in $\tilde{\G}_v(\FF_{v})$, instead of the $q$-divisors of $P$ in $\G(k_v)$. Problem  \ref{probW} was even formulated quite
at the same time than Problem 1.  It is also related
to Problem \ref{ProbS} in abelian varieties, in the case when $P$ is the zero point. 
The main result about Wong's question is the following.

\begin{thm}[Wong 2000]
Let $\A$ be an abelian variety defined over a number field $k$, let $U=\{P\}$, with $P\in \A(k)$ and let $q$ be a positive integer.
Assume that one of the following conditions hold

\medskip
\begin{enumerate}
\item[a)] $H^1(\Gal(k(\A[q])/k),\A[q])=0$;
\item[b)] $\A$ is an elliptic curve and $q=p$.
\end{enumerate}
\medskip

\noindent If the image of  $P$ in $\tilde{\A}(\FF_v)$ is the $q$-th power of an element of $\tilde{\A}(\FF_v)$ for a set of places
$v$ of density 1, then $P$ is the $q$-th power of a point in $\A(k)$.
\end{thm}

 \item[iv.]  Let $\E$ be an elliptic curve defined over $k$. 
It is well-known that there is a close connection between the existence of a $k$-rational  torsion point of order $p$ in $\E$
and the existence of a $k$-rational isogeny $\phi:\E\rightarrow \E$ of degree $p$. Therefore Question \ref{probP} (as well as the other mentioned problems) is also linked to the
following  local-global problem for existence of isogenies of prime degree in elliptic curves. 

\begin{Problem}[Sutherland, 2012] \label{probS}
Let $k$ be a number field and $\E$ be an elliptic curve defined over $k$.
Assume that $\E$ admits a $k_v$-rational isogeny of degree $p$, for all places $v$ of $k$.
 Does $\E$ admit a $k$-rational isogeny of degree $p$? 
\end{Problem}

\noindent As we recalled in Section \ref{counter1}, the connection between the existence of isogenies and Problem \ref{prob1}
is also underlined in \cite{DZ3}. The converse of Problem \ref{probS} is trivially true.
Sutherland proved the following result.

\begin{thm}[Sutherland, 2012] \label{sut}
Let $p$ be a prime number. Assume that $\sqrt{\left(\frac{-1}{p}\right)p}\notin k$ and that $\E$ admits a
rational isogeny of degree $p$ locally at a set of primes with density one. Then $\E$ admits
an isogeny of degree $p$ at a quadratic extension of $k$. If $p\equiv 1 \modn 4)$ or $p<7$,
then $\E$ admits a $k$-rational isogeny of degree $p$.  
\end{thm}

\noindent When $k=\QQ$, he furthermore showed that  the question has an affirmative answer for
every prime $p\neq 7$. If $p=7$ there exists only one counterexample up to isomorphism.
The counterexample is given by the elliptic curve with equation

$$y^2+xy=x^3-x^2-107x-379,$$

\noindent admitting an isogeny of degree 7 locally at every prime of good reduction and over $\RR$, but
admitting no rational isogenies of degree 7. Anyway, according to Theorem \ref{sut}, the curve admits
a $k$-rational isogeny over a quadratic extension $k$ of $\QQ$. In addition, when $k$ is a number field
containing the quadratic subfield of $\QQ(\z_p)$, the author gives a classification of the curves
for which the principle fails. \par The study was completed in \cite{BC} by Banwait and Cremona for
number fields $k$ that do not contain the quadratic subfield of $\QQ(\z_p)$. In particular they show
all possible elliptic curves such that the principle fails when $k$ is a quadratic extension.
\par In \cite{Anni} Anni gives an upper
bound for the primes $p$ such that the local-global divisibility for existence of isogenies of degree
$p$ may fail in elliptic curves over a number field $k$. This bound depends
only on the degree of $k$ and on its discriminant.
\par In the recent paper \cite{Vog}, Vogt considers Problem \ref{probS} for
rational isogenies of arbitrary degree $q$. In particular  he shows 
that for a fixed number field $k$ and a fixed positive integer $q$ there are only finitely many non-isomorphic elliptic curves for which the local-global existence of a rational isogeny of degree $q$ fails.

\item[v.] Having investigated about the vanishing of two subgroups of 
$H^1(G,\G[p^l])$, we have to recall that a very interesting question is about the vanishing of this
group itself.  In the case of elliptic curves, this problem has been especially investigated
in \cite{Cha} and in the mentioned \cite{LW}, in which the authors prove the following statement.

\begin{thm}[Lawson, Wuthrich, 2016]
Let $\E$ be an elliptic curve defined over $\QQ$. The group $H^1(G,\E[p])$ is trivial except in the following cases

\medskip
\begin{enumerate}
\item[1.] $p = 3$, there is a rational point of order $3$ on $\E$, and there are no other isogenies of degree $3$ from
$\E$ that are defined over $\QQ$;
\item[2.] $p = 5$ and the quadratic twist of $\E$ by $D = 5$ has a rational point of order $5$, but no other isogenies
of degree $5$ defined over $\QQ$;
\item[3.] $p = 11$ and $\E$ is the curve labeled as 121c2 in Cremona’s label, given by the global minimal
equation $y^2+xy=x^3+x^2-3632x+82757$.
\end{enumerate}
\medskip

\noindent In each of these cases, $H^1(G,\E[p])$ has $p$ elements.
\end{thm}

\noindent For a general commutative algebraic group $\G$, sufficient conditions to the
vanishing of $H^1(G,\G[p])$ are given by Nori in \cite[Theorem E]{Nori}. 

\begin{thm}[Nori, 1972]
There exists a constant $c(n)$ depending only on $n$ such that if  
$p>c(n)$ and $G\leq \GL_n(\FF_p)$ acts semisimply on $\FF_p^n$, then
$$H^1(G,\FF_p^n)=0.$$
\end{thm}

\noindent Many other authors have investigated about the vanishing of the group
$H^1(\Gamma,M)$, where $\Gamma$ is a group and $M$ is a $\Gamma$-module
(see for examples among others \cite{CPS}, \cite{CPS2}, \cite{Georgia}).

 \item[vi.] Another question, that is not a local-global one, but it is strongly related to Problem 1 (and
consequently to Problem 2) is the classification of all $q$-division fields $k(\G[q])$, for a fixed integer $q$.
In fact, information about the extension  $k(\G[q])/k$ provides information about
the Galois group $G=\Gal(k(\G[q])/k)$ and then about the local cohomology group
$H^1_{\loc}(G,\G[q])$ and the Tate-Shafarevich group $\Sha(k,\G[q])$. In particular in \cite{Pal2} the interest of classifying
all elliptic curves such that $\QQ(\E[3])=\QQ(\z_3)$ was motivated by the possible applications
to Problem 1. Anyway, independently from the local-global divisibility, 
an interesting question  is to understand when the field $k(\E[p])$ is as small as possible, i.e. $k(\E[p])=k(\z_p)$.
We have already mentioned that  Merel and Stein \cite{MS} and Rebolledo \cite{Reb} 
proved that $\QQ(\E[p])=\QQ(\z_p)$ implies 
$p\in \{2,3,5\}$ or $p>1000$.  The curves with $\QQ(\E[2])=\QQ(\z_2)$ (resp. $k(\E[2])=k(\z_2)$) are
the ones with two rational (resp. $k$-rational)  torsion points of order 2, that are linearly independent.
All the curves with $\QQ(\E[3])=\QQ(\z_3)$ (resp. $k(\E[3])=k(\z_3)$)  
are shown in \cite{Pal2} (resp. \cite{BP2}). All the
curves with $\QQ(\E[5])=\QQ(\z_5)$ were lately classified in \cite{GJ} by Gonz\'{a}lez-Jim\'{e}nez and Lozano-Robledo.
\noindent They also proved that if $\QQ(\E[q])=\QQ(\z_q)$, for any integer $q$, then
$q\in \{2,3,4,5\}$ and describe the family of elliptic curves such that $\QQ(\E[4])=\QQ(\z_4)$.
Moreover, they studied some properties of  the extension $\QQ(\E[q])/\QQ$ in the case when it is abelian. In particular
they describe all possible abelian Galois groups $\Gal(\QQ(\E[q])/\QQ)$ and prove the following statement.

\begin{thm}[Gonz\'{a}lez-Jim\'{e}nez, Lozano-Robledo, 2016]
Let $\E$ be an elliptic curve defined over $\QQ$ and let $q$ be a positive integer.
Assume that  $\QQ(\E[q])/\QQ$ is abelian. Then $n\in\{2,3,4,5,6,8\}$.
\end{thm}

 A classification of all number fields $k(\E[q])$, for $q\in \{3,4\}$ is given in \cite{BP2} (see also \cite{BP}
for number fields $\QQ(\E[3])$ and \cite{Pal18} for number fields $k(\E[5])$, where $\E$ is an elliptic curve with complex multiplication with Weierstrass form $y^2=x^3+bx$ or
$y^2=x^3+c$, where $b,c\in \QQ$). 
In the same paper a new set of generators is provided for the extension $k(\E[q])$, when $q$ is an odd number.
Let   $\z_q$ be a primitive $q$-th root of the unity as above
and $P_1=(x_1,y_1)$ and $P_2=(x_2,y_2)$  two $q$-torsion points of $\E$ forming a basis
of $\E[q]$. 
Then $$k(\E[q])=k(x_1,\z_q,y_2).$$  In addition, if $q=p^l$, with $p$ odd, then
$k(\E[p^l])=k(x_1,\z_p,y_2),$ for every $l$ \cite{DivPal}, where $\z_p$ is a primitive $p$-th root of the unity
and $P_1=(x_1,y_1)$ and $P_2=(x_2,y_2)$ two $p^l$-torsion points of $\E$ forming a basis
of $\E[p^l]$. For some other information about $q$-division fields $k(\E[q])$, see also \cite{Ade}.

\end{enumerate}

\section{Declarations}

\begin{description}

\item[Funding] Not applicable.
\item[Conflicts of interest/Competing interests] None.
\item[Availability of data and material] Not applicable.
\item[Code availability] Not applicable.
\end{description}



\vskip 1.5cm


\end{document}